\documentclass[11pt,reqno]{amsart}
\usepackage{color}
\usepackage[active]{srcltx}
\usepackage{a4wide}
\usepackage{amssymb, amsmath, mathtools}
\usepackage{graphicx}
\usepackage{float}
\usepackage[active]{srcltx}
\usepackage[ansinew]{inputenc}
\usepackage{hyperref}

\theoremstyle{plain}
\newtheorem{theorem}{Theorem}[section]
\newtheorem{maintheorem}{Theorem}
\newtheorem{lemma}[theorem]{Lemma}
\newtheorem{proposition}[theorem]{Proposition}

\newtheorem{mainproposition}[maintheorem]{Proposition}

\theoremstyle{remark}

\newtheorem{definition}{Definition}

\newtheorem{remark}[theorem]{Remark}

\newcommand{\field}[1]{\mathbb{#1}}
\newcommand{\RR}{\field{R}}
\newcommand{\CC}{\textbf{C}}
\newcommand{\ZZ}{\field{Z}}
\newcommand{\dpt}{\displaystyle}
\newcommand{\NN}{\field{N}}

\newcommand{\EU}{{\mathbb{S}}}

\newcommand{\loc}{\textnormal{loc}}

\newcommand{\In}{{\text{In}}}
\newcommand{\Out}{{\text{Out}}}

\numberwithin{equation}{section}

\usepackage[pagewise]{lineno}

\begin{document}

\title[Rank-one strange attractors vs. Heteroclinic tangles]{Rank-one strange attractors  \emph{versus} \\ Heteroclinic tangles}
\author[Alexandre Rodrigues]{Alexandre A. P. Rodrigues \\ Centro de Matem\'atica da Univ. do Porto \\ Rua do Campo Alegre, 687,  4169-007 Porto,  Portugal }
\address{Alexandre Rodrigues \\ Centro de Matem\'atica da Univ. do Porto \\ Rua do Campo Alegre, 687 \\ 4169-007 Porto \\ Portugal}
\email{alexandre.rodrigues@fc.up.pt}

\date{\today}

\thanks{AR was partially supported by CMUP, which is financed by national funds through FCT -- Fundação para a Ci\^encia e Tecnologia, I.P., under the project with reference UIDB/00144/2020.  The author also acknowledges financial support from Program INVESTIGADOR FCT (IF/00107/2015).}


\begin{abstract}
We present a mechanism for the emergence of strange attractors (observable chaos) in a two-parameter periodically-perturbed  family of differential equations on the plane. The two parameters are independent and act on different ways in the invariant manifolds of consecutive saddles in the cycle. 
When both parameters are zero, the flow exhibits an attracting heteroclinic   cycle associated to two equilibria. The first parameter makes the two-dimensional invariant manifolds of consecutive saddles in the cycle to pull apart; the second forces transverse intersection. These relative positions may be determined using the Melnikov method.

Extending the previous theory on the field, we prove the existence of many complicated dynamical objects in the two-parameter family, ranging from  ``large'' strange attractors supporting SRB (Sinai-Ruelle-Bowen) measures to superstable sinks and H\'enon-type attractors.   We   draw  a plausible bifurcation diagram associated to the problem under consideration and we show  that the occurrence of  heteroclinic tangles is a \emph{prevalent} phenomenon.

\end{abstract}

\maketitle
\setcounter{tocdepth}{1}

\section{Introduction}\label{intro}

Periodically perturbed homoclinic cycles have been studied extensively in history. The topic has occupied a center position of the chaos theory since the time of  Poincar\'e \cite{Poincare}. Literature on the mathematical analysis and on numerical simulations is rather abundant. We mention a few that are closely related to this paper: the theory of Smale's horseshoes \cite{Smale} and its applications to differential equations through the Melnikov method \cite{Melnikov}; the  work from Shilnikov's team \cite{Shilnikov70, Shilnikov_book_2}, those from Chow and Hale's school \cite{CH}, concerning chaos and heteroclinic bifurcations in autonomous differential equations and those from Wang, Ott, Oksasoglu and Young concerning rank-one strange attractors \cite{WY, WY2003, WY2006}.

In this paper, we study the dynamics of strange attractors (sustainable chaos) in periodically perturbed  differential equations with two heteroclinic connections associated to two dissipative saddles. 
An explicit formula for the first return map to a cross section is  derived. By extending the theory developed for the one loop case in \cite{WO, WY2003},   we obtain a generic overview on various admissible dynamical scenarios for the associated non-wandering sets. 
We state precise hypotheses  that imply the existence of   observable chaos and  sinks for a set of forcing amplitudes with positive Lebesgue measure.

Motivated by  bifurcation scenarios involving homoclinic cycles \cite{ChOW2013}, we prove the existence of many complicated dynamical objects for a given equation, ranging from an attracting torus of quasi-periodic solutions, Newhouse sinks and H\'enon-like attractors, to rank-one strange attractors with Sinai-Ruelle-Bowen (SRB) measures.  The theory developed in this paper is explicitly applicable to the analysis of various specific differential equations and the results obtained go beyond the capacity of the classical Birkhoff-Melnikov-Smale method \cite{GH}. 

Our purpose in writing this paper is not only to point out the range of phenomena that can occur when simple non-linear equations are periodically forced, but to bring to the foreground the methods that have allowed us to reach these conclusions in a  straightforward manner. These techniques are not limited to the system considered here.

\subsection*{Structure of the article}
This article is organised as follows: in Section \ref{s:setting} we describe the problem and we refer some related literature on the topic.  In Section \ref{s: main results}, we  state the main results of this research and we explain how they fit in the literature. In Sections \ref{s:definitions}, \ref{s: theory1}, we introduce some basic concepts for the understanding of this article and we review the theory of rank-one attractors stated in \cite{WY2003}.
The proof of the main results   will be performed in Section \ref{s: Proof Theorems A and B} and \ref{Proof_ThC}, after the precise computation of suitable first return maps to cross sections in Section \ref{first return map}.

In Sections \ref{s: proof Th D} and \ref{Proof_ThE} we prove the results related to the heteroclinic tangle. We also show that the existence of heteroclinic tangles is a prevalent phenomenon in a bifurcation diagram in Section \ref{Proof_ThF}. In Section \ref{s: Melnikov}, we describe explicitly the expressions requested  by the Melnikov integral in order to satisfy Hypotheses \textbf{(P6)--(P7)} of Section \ref{s:setting}. We discuss the consequences of our findings in Section \ref{s: discussion}. 

  In Appendix \ref{app1}, we list the main notation for constants and terminology in order of appearence. 
We use  the setting of \cite{OW2010, TW12} because we are interested in the admissible families that are obtained by passing to the singular limits of families of rank one maps with/without logarithmic singularities. 

We have endeavoured to make a self contained exposition bringing together all topics related to the proofs. We have drawn illustrative figures to make the paper easily readable.

\section{Setting}
\label{s:setting}
\subsection{Starting point}
\label{ss:setting}
Let $(x, y)\in \RR^2$ be the phase variables and $t$ be the independent variable. We start with the following autonomous differential equation:
\begin{equation}
\label{general_R2}
\left\{
\begin{array}{l}
\dot x= g_1(x,y)\\
\dot y= g_2(x,y) 
\end{array}
\right.
\end{equation}
where $g_1 $ and $g_2$  are analytic functions defined on an open domain $\mathcal{V} \subset \RR^2$ and $\dot{x}=  \frac{dx}{dt}$, $\dot{y}=  \frac{dy}{dt}$. We assume that \eqref{general_R2} has two hyperbolic equilibria in $\mathcal{V}$, say $O_1=(x_1, y_1)$ and $O_2=(x_2, y_2)$ (see Figure \ref{gs1}). Let $-c_1<0<e_1$ be the eigenvalues of the Jacobian matrix of \eqref{general_R2}  at $O_1$, and $\bar{u}({c_1}), \bar{u}({e_1})$ be their  associate unit eigenvectors. 

Analogously, let $-c_2<0<e_2$ and $\bar{u}({c_2}), \bar{u}({e_2})$ be the corresponding  eigenvalues and unit eigenvectors for the jacobian matrix of \eqref{general_R2} at $O_2$. We assume that both $O_1$ and $O_2$ satisfy the following conditions:
\bigbreak
\begin{enumerate}
 \item[\textbf{(P1)}] \label{B1}  (Dissipativeness) $c_1>e_1$ and $c_2>e_2$.
 \bigbreak
 \item[\textbf{(P2)}] \label{B2} (Non-resonant condition) For $i\in \{1,2\}$, there exist $d_1, \tilde{d_1}, d_2,  \tilde{d_2}\in \RR^+$ such that for all $m,n \in \NN$, the following inequality holds:
$$
|m\, c_i-n\, e_i|> d_i (\, |m|+|n|\, )^{-\tilde{d_i}}.
$$

 \medbreak
\end{enumerate}
\bigbreak

\begin{enumerate}
 \item[\textbf{(P3)}] \label{B3} (Heteroclinic cycle) The system \eqref{general_R2} has two heteroclinic solutions in $\mathcal{V}$: one from $O_1$ to $O_2$, which we denote by $\ell_1 = \{ (a_1(t),b_1(t)), t \in \RR\}$; and the other from $O_2$ to $O_1$, which we denote by $\ell_2 = \{ (a_2(t),b_2(t)), t \in \RR\}$, forming a heteroclinic cycle (see Figure \ref{gs1}).
\end{enumerate}
\bigbreak
For $\varepsilon>0$ sufficiently small,   we add two forcing terms to  \eqref{general_R2}  of the type:
\begin{equation}
\label{general_R2_perturbed}
\left\{
\begin{array}{l}
\dot x= g_1(x,y)+\mu_1 P_1(x,y,\omega\,  t)+\mu_2 Q_1(x,y,\omega\,  t)\\
\dot y= g_2(x,y) +\mu_1 P_2(x,y,\omega\,  t)+\mu_2 Q_2(x,y,\omega\,  t)
\end{array}
\right.
\end{equation}
where $\omega>0$, $ \mu_1, \mu_2$ are small independent parameters in $[0, \varepsilon] $ and $$P_1(x, y,t), P_2(x, y,t), Q_1(x, y,t),Q_2 (x, y,t) : \quad \mathcal{V}\times \RR  \longrightarrow  \RR$$ are $C^4$. We also assume that:
\bigbreak

\begin{enumerate}
 \item[\textbf{(P4)}](Periodic perturbations)  \label{B4}  For $j\in \{1,2\}$, there exists $T>0$ such that $$\forall x, y\in \mathcal{V}, \quad P_j(x,y,t+T)=P_j(x,y,t)\quad \text{and}\quad Q_j(x,y,t+T)=Q_j(x,y,t).$$ \bigbreak 
 \item[\textbf{(P5)}] \label{B5}  The value of $P_1(x, y,t)$, $P_2(x, y,t)$, $Q_1(x, y,t)$ and $Q_2(x, y,t)$ and their first derivatives with respect to $x$ and $y$ are all zero at $O_1$ and $O_2$ for all $t$.
\end{enumerate}

\begin{figure}[ht]
\begin{center}
\includegraphics[height=5.5cm]{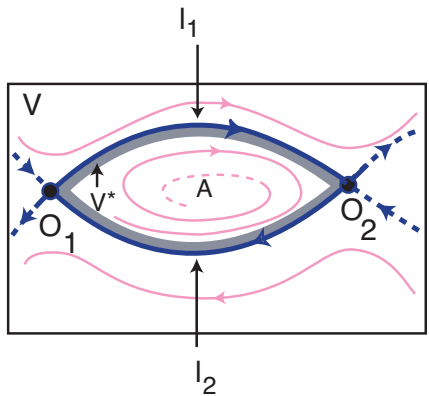}
\end{center}
\caption{\small  The dynamics of  $\eqref{general_R2}$ defined $\mathcal{V}\subset \RR^2$ is  governed by the existence of an heteroclinic cycle associated to $O_1$ and $O_2$. $\ell_1, \ell_2$: heteroclinic connections; $\mathcal{V}^\star$: inner basin of attraction of the cycle (absorbing domain); $\mathcal{A}:$ region limited by the cycle. }
\label{gs1}
\end{figure}
\bigbreak
In $\mathcal{V}$, the heteroclinic cycle $\ell_1\cup \ell_2\cup\{O_1, O_2\}$ limits a region that we call   $\mathcal{A}$. For $\mu_1=\mu_2=0$, there is an open set $\emptyset \neq \mathcal{V}^\star$ of $ \mathcal{A}$ such that the $\omega-$limit of all solutions starting in $\mathcal{V}^\star$ is $\ell_1\cup \ell_2$. In other words, the cycle $\ell_1\cup \ell_2\cup\{O_1, O_2\}$ is asymptotically stable ``by inside'' (configuration similar to that of Takens \cite{Takens94}\footnote{This configuration is also called the ``\emph{attracting Bowen eye}''.}).
\subsection{The lift}
\label{ss:lift}
We now introduce an angular variable $\theta \in \EU^1=\RR/{\ZZ T }$ to rewrite \eqref{general_R2_perturbed} as
\begin{equation}
\label{general_R3}
\left\{
\begin{array}{l}
\dot x= g_1(x,y)+\mu_1 P_1(x,y,\theta)+\mu_2 Q_1(x,y,\theta)\\
\dot y= g_2(x,y)+\mu_1 P_2(x,y,\theta)+\mu_2 Q_2(x,y,\theta) \\
\dot \theta= \omega \\
\end{array}
\right.
\end{equation}
or, in matricial notation, by:

\begin{equation}
\label{general_R3_matrix}
\left(\begin{array}{c} \dot{x}  \\ \dot y \\ \dot \theta \end{array}\right)= \left(\begin{array}{c} g_1(x,y)   \\ g_2(x,y) \\ \omega \end{array}\right)+ \mu_1  \left(\begin{array}{c} P_1(x,y, \theta)  \\ P_2(x,y, \theta) \\ 0 \end{array}\right) + \mu_2  \left(\begin{array}{c} Q_1(x,y, \theta)  \\ Q_2(x,y, \theta) \\ 0 \end{array}\right).
\end{equation}

\begin{figure}[ht]
\begin{center}
\includegraphics[height=4.9cm]{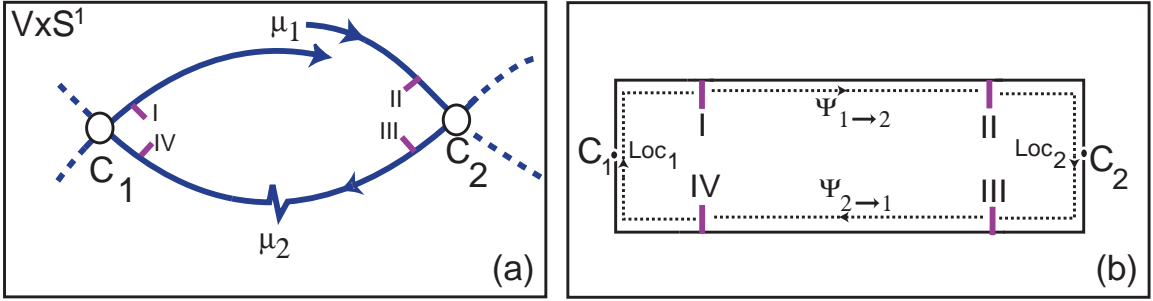}
\end{center}
\caption{\small (a) Scheme of the effects of the parameters $\mu_1, \mu_2$ on the equations \eqref{general_R3}. (b) Sketch of the local and transition maps. $I, II, III$ and $IV$ represent cross sections $\Out(\CC_1)$, $\In(\CC_2)$, $\Out(\CC_2)$ and $\In(\CC_1)$, respectively.  }
\label{notation1}
\end{figure}

The vector field associated to  equation \eqref{general_R3} will be denoted by $f_{(\mu_1, \mu_2)}$.
In the  phase space of \eqref{general_R3}, say   $\mathbb{V} = \mathcal{V} \times \EU^1$, for $\mu_1=\mu_2=0$, there is an attracting heteroclinic cycle $\Gamma$  between two hyperbolic periodic solutions, say $\CC_1\cup \CC_2$, connected by two manifolds diffeomorphic to tori $\mathcal{L}_1$ and $ \mathcal{L}_2$. The periodic solution $\CC_i$ is the lift   of $O_i$, $i=1,2$. Set  $\mathbb{A}= \mathcal{A}  \times \EU^1$.

 In the $(x, y, \theta)$--space, let $V_1$ and $V_2$ be two hollow cylinders around $\CC_1$ and $\CC_2$, respectively, where a local normal form may be defined. Let $\Out(\CC_1)$ and $\Out(\CC_2)$ be two sections (planes) transverse to $\mathcal{L}_1\cup \mathcal{L}_2$ where all initial conditions go outside $V_1$ and $V_2$ in positive time, respectively. Analogously,  let $\In(\CC_1)$ and $\In(\CC_2)$ be two sections (planes) transverse to $\mathcal{L}_1\cup \mathcal{L}_2$ where where all initial conditions go inside $V_1$ and $V_2$ in positive time, respectively.

\subsection{Parameters effects}
\label{ss:parameters}
Concerning the addition of the non-zero perturbing terms whose magnitude is governed by $\mu_1$ and $\mu_2$, the effect on the dynamics of \eqref{general_R3} differs from the type of  intersection between the invariant manifolds of $ \CC_1$ and $\CC_2$ as follows:
\medbreak
\begin{description}
\item[Case 1] $W^u(\CC_1) \pitchfork W^s(\CC_2)$ and $W^u(\CC_2) \pitchfork W^s(\CC_1)$ \\
\item[Case 2] $W^u(\CC_1) \cap W^s(\CC_2)= \emptyset$ and $W^u(\CC_2) \pitchfork W^s(\CC_1)$\\
\item[Case 3] $W^u(\CC_1) \pitchfork W^s(\CC_2)$ and $W^u(\CC_2) \cap W^s(\CC_1)= \emptyset$\\
\item[Case 4] $W^u(\CC_1) \cap W^s(\CC_2)= \emptyset$ and $W^u(\CC_2) \cap W^s(\CC_1)= \emptyset,$
\end{description}
\medbreak
\noindent where $A\pitchfork B$ means that the manifolds $A$ and $B$ intersect transversely.  
In Table 1, we identify these four cases.

\begin{table}[htb]
\label{Table1}
\begin{center}
\begin{tabular}{|c|c|c|} \hline 
&&\\
{Configuration}  & \qquad  $W^u(\CC_1) \pitchfork W^s(\CC_2)$ \qquad  \qquad &\qquad  $W^u(\CC_1) \cap W^s(\CC_2)=\emptyset$ \qquad \\
&&\\
\hline \hline
&&\\
$W^u(\CC_2) \pitchfork W^s(\CC_1)$ &Case 1 &Case 2 \\  &&\\
 \hline \hline  &&\\
 $W^u(\CC_2) \cap W^s(\CC_1)=\emptyset$ &Case 3 & Case 4 \\
&& \\ \hline

\hline

\end{tabular}
\end{center}
\label{notationA}
\bigskip
\caption{\small Four generic different cases for the dynamics of \eqref{general_R3}. }
\end{table} 
\begin{remark} Since $\CC_1$ and $\CC_2$ are hyperbolic, they persist for $\mu_1, \mu_2 > 0$ small.
For $\mu_1>0$, when we consider empty intersection of the invariant manifolds, we mean  $W^u(\CC_1)$  enters the absorbing domain $\mathbb{V}^\star$, otherwise there are no guarantees that the set of non-wandering points is non-empty.  
\end{remark}

For $ (\mu_1, \mu_2) \in [0, \varepsilon]\times [0, \varepsilon]$, let    $\mathcal{F}_{(\mu_1, \mu_2)}$ and   $\mathcal{G}_{(\mu_1, \mu_2)}$ be the return maps to the cross sections $\Out(\CC_1)$  and $\Out(\CC_2)$, respectively.     Denote
$$
\Omega\left(\mathcal{F}_{(\mu_1, \mu_2)}\right) =\left\{X \in \Out(\CC_1):\quad  \mathcal{F}_{(\mu_1, \mu_2)}^n (X) \in \Out(\CC_1),\quad  \forall n \in \NN\right\} 
$$
and
$$
\Lambda \left(\mathcal{F}_{(\mu_1, \mu_2)}\right) = \bigcap_{n\in \NN}  \mathcal{F}_{(\mu_1, \mu_2)}^n(\Omega_\mu) .
$$
 The set $\Omega(\mathcal{F}_{(\mu_1, \mu_2)})$ represents all solutions of \eqref{general_R3} that stay around the unforced heteroclinic loop $\mathcal{L}_1\cup \mathcal{L}_2$  in forward time and $\Lambda (\mathcal{F}_{(\mu_1, \mu_2)})$ represents all solutions that stay around  $\mathcal{L}_1\cup \mathcal{L}_2$, for all time.  Analogously we define $\Omega\left(\mathcal{G}_{(\mu_1, \mu_2)}\right)$ and $\Lambda \left(\mathcal{G}_{(\mu_1, \mu_2)}\right) $, replacing $ \Out(\CC_1)$ by $ \Out(\CC_2)$.
With respect to the effect of the perturbations governed by $\mu_1$ and $\mu_2$, both act independently and we state the following hypotheses ($A\equiv_\mathbb{V^\star} B$ means that the manifolds $A$ and $B$ coincide within $ {\mathbb{V}^\star}$):
\bigbreak
\begin{enumerate}
 \item[\textbf{(P6a)}] If $\mu_1>0$ and $\mu_2=0$,  then $W^u(\CC_1) \cap W^s(\CC_2)=\emptyset$ and $W^u(\CC_2) \equiv_{\mathbb{V}^\star} W^s(\CC_1)$. 
  \medbreak 
 \item[\textbf{(P6b)}]  If $\mu_2>0$ and $\mu_1=0$,  then $W^u(\CC_1)\equiv_{\mathbb{V}^\star}  W^s(\CC_2) $ and $W^u(\CC_2) \pitchfork W^s(\CC_1)$.
\end{enumerate}
\bigbreak
 When we refer to \textbf{(P6)}, we refer to \textbf{(P6a)} and \textbf{(P6b)}.  For $\mu_1, \mu_2 \in [0, \varepsilon]$, in the local coordinates of Subsection \ref{coords}, the flow associated to  \eqref{general_R3} induces the $C^3$--embeddings $$\Psi_{1 \rightarrow 2}: \Out(\CC_1) \rightarrow  \In(\CC_2) \quad \text{and} \quad \Psi_{2 \rightarrow 1}: \Out(\CC_2) \rightarrow  \In(\CC_1)$$ of the form\footnote{These hypotheses will be clear in Section \ref{first return map}.}: 
\begin{enumerate}
 \item[\textbf{(P7a)}]  
 \begin{equation}
 \label{P7a_cond}
 \Psi_{1 \rightarrow 2}\left(y_1^{(1)}, \theta^{(1)}\right) = \left[ c_1 y_1^{(1)}+\mu_1\Phi_1\left(y_1^{(1)}, \theta^{(1)} \right); \quad \theta^{(1)} + \xi_1 +\mu_1\Psi_1\left(y_1^{(1)}, \theta^{(1)} \right)\right]
 \end{equation}
 \end{enumerate}
  where $c_1 \neq 0$, $\xi_1 \in \RR$, $\Psi_1: \Out(\CC_1)\rightarrow\RR$ is $C^1$ and $ \Phi_1: \Out(\CC_1)\rightarrow\RR^+$ is $C^3$,  non-constant and has a finite number of non degenerate critical points. Furthermore,
 \begin{enumerate}
   \medbreak 
 \item[\textbf{(P7b)}]  
\begin{equation}
 \label{P7b_cond}
 \Psi_{2 \rightarrow 1}\left(y_1^{(2)}, \theta^{(2)}\right) = \left[ c_2 y_1^{(2)}+\mu_2\Phi_2\left(y_1^{(2)}, \theta^{(2)} \right); \quad \theta^{(2)}  + \xi_2 +\mu_2\Psi_2\left(y_1^{(2)}, \theta^{(2)} \right)\right]
\end{equation}
  
\end{enumerate}
where $c_2 \neq 0$, $\xi_2 \in \RR$, $\Psi_2, \Phi: \Out(\CC_2)\rightarrow\RR$ are $C^1$ and $ \Phi_2: \Out(\CC_2)\rightarrow\RR$  has at least two non degenerate zeros.  


\begin{remark}\label{remark 2.3}
When there is no risk of misunderstanding, we identify $\Phi_1(\theta) \equiv \Phi_1(0, \theta) $ and $ \Phi_2(\theta)\equiv\Phi_2(0, \theta),$ where $\theta\in \EU^1$.
\end{remark}

\subsection{Literature on the topic and the goal of this article}
Case 1 has been studied  in  \cite{ChOW2013, LR2017, RLA}; the authors found a sequence of suspended horseshoes accumulating on the cycle, homoclinic tangencies and Newhouse phenomena giving rise to sinks and H\'enon-type strange attractors.  Their results do not depend on the frequency  $\omega$. 
Case 4  has been discussed in \cite{Mohapatra}, where the authors proved the existence of an attracting torus for $\omega \approx 0$  and rank-one attractors for $\omega \gg 1$ (sufficiently large). 
Combining  the techniques developed in \cite{ChOW2013, Mohapatra}, in this paper we deal with Case 2 illustrated in Figure \ref{notation1}. Case 3 has a similar treatment. We also provide complementary results for Cases 1 and 4.

\section{Main results}
\label{s: main results}

  Once for all, let us fix $\varepsilon>0$ small. Denote by $\mathfrak{X}_{\Gamma}^4(\mathbb{V} )$ the  two-parameter family of $C^4$--vector fields \eqref{general_R3}   satisfying conditions \textbf{(P1)--(P7)}. Before going further, we set two positive constants that will be used in the sequel:
\begin{equation}
\label{formulas3.1}
K_F= \frac{1}{e_2 }+ \frac{c_2/e_2}{e_1 }\qquad \text{and}\qquad K_G= \frac{1}{e_1 }+ \frac{c_1/e_1}{e_2 }.
\end{equation}

Our first result strongly  relies  on the global map $\Psi_{1 \rightarrow 2}$ from $\Out(\CC_1)$ to  $ \In(\CC_2)$ (cf. \eqref{P7a_cond} and Remark \ref{remark 2.3}).
\bigbreak

\begin{maintheorem}\label{thm:B}
Let $f_{(\mu_1,0)} \in\mathfrak{X}_{\Gamma}^4(\mathbb{V})$, with $\mu_1>0$. If $\omega$ is such that 
$$ \dpt
 \omega\times  \sup_{\theta\in \EU^1}\left( \frac{\Phi_1'(\theta)}{\Phi_1(\theta)}\right)<\frac{1}{K_F},
$$
   then there is an invariant closed curve $\mathcal{C}\subset \Out(\CC_1)$ as the maximal attractor for the map $\mathcal{F}_{(\mu_1, 0)}$.  This closed curve is not contractible on $\Out(\CC_1)$.  
\end{maintheorem}
 \bigbreak

Although we use the Theory of Rank-one attractors to prove Theorem \ref{thm:B}, this result may be shown using  the Afraimovich's Annulus Principle~\cite{AHL2001}. 
The  curve $ \mathcal{C}\subset \Out(\CC_1)$ is \emph{globally attracting} in the sense that, for every $X\in \mathbb{V}^\star$, there exists a point $X_0\in \mathcal{C}$ such that 
$$
\lim_{n\rightarrow +\infty} \left\| \mathcal{F}_{(\mu_1, 0)}^n (X) -\mathcal{F}_{(\mu_1, 0)}^n (X_0) \right\| =0,
$$
where $\|\star \|$ is the usual norm induced from $\RR^2$.
The attracting invariant curve for  $\mathcal{F}_{(\mu_1,0)}$ is  the graph of a smooth map and corresponds to an attracting two-dimensional torus for the flow of \eqref{general_R3}.   When $\mu_1>0$ and $\mu_2=0$, one branch of $W^u(\CC_1)$ accumulates on the torus.

The dynamics of $\mathcal{F}_{(\mu_1, 0)}$ induces on  $\mathcal{C}$  a circle map. Indeed, for any given interval of unit length $[0, \varepsilon]$, there is a positive measure set $\Delta \subset [0, \varepsilon]$ so that the rotation number of $\mathcal{F}_{(\mu_1, 0)}|_\mathcal{C}$ is irrational  if and only if $\mu_1 \in \Delta$.
This implies the existence of a set of positive Lebesgue measure (in the bifurcation parameter)  for which the torus has a dense orbit, \emph{i.e.} the torus is a \emph{minimal attractor} \cite{Herman}.

\bigbreak
\begin{maintheorem}\label{thm:F}
Let $f_{(\mu_1,0)} \in\mathfrak{X}_{\Gamma}^4(\mathbb{V})$, with $\mu_1>0$.
There exists  $\omega^\star>0$ such that for  all $\omega>\omega^\star$, there is a subset of positive Lebesgue measure $\Delta\subset [0, \varepsilon]$ for which the map $\mathcal{F}_{(\mu_1, 0)}$ with $\mu_1\in \Delta$, exhibits rank-one strange attractors with an ergodic SRB measure. 
 \end{maintheorem}

\bigbreak

 The existence of rank-one strange attractors for $ \mathcal{F}_\mu$ is an \emph{abundant phenomenon} in the terminology of \cite{MV93}. Furthermore, these attractors are  \emph{``large''} according to \cite{BST98}, \emph{i.e.} their non-wandering points wind around an entire non-contractible annulus.
 These  strange attractors  have strong statistical properties that will be made precise in Section \ref{s: theory1} (see also \cite{WY2006}).

The proof of Theorems \ref{thm:B} and \ref{thm:F}  is performed  in Subsections \ref{Proof_ThA} and \ref{Proof_ThB} by reducing the analysis of the two-dimensional map $\mathcal{F}_{(\mu_1, 0)}$ to the dynamics of a one-dimensional map, via the \emph{Rank-one attractors}'s theory. 

\bigbreak
\begin{maintheorem}\label{thm:C}
Let $f_{(\mu_1,0)} \in\mathfrak{X}_{\Gamma}^4(\mathbb{V})$, with $\mu_1>0$.
Under an open technical hypothesis \textbf{(TH)} on the space of parameters, for $\omega\gg 1$ there exists a sequence of real numbers converging to zero, say $\left(\mu_{1,n}\right)_{n\in \NN}$, for which the flow of \eqref{general_R3} exhibits a periodic sink. 
\end{maintheorem}
 This sink does not follow from the Newhouse theory \cite{Newhouse79}; it is \emph{superstable} in the sense that one of its Floquet multipliers is very close to zero. The open condition  \textbf{(TH)}  stated in Theorem \ref{thm:C} is technical and depends on a specific variable in the definition of Misiurewicz-type map in Subsection \ref{Misiurewicz-type map}. It will be clear in Section \ref{Proof_ThC}, where the proof of the result is performed.
 
 \bigbreak

 The next two results concern the case $\mu_1=0$ and $\mu_2>0$, where $W^u(\CC_1)\equiv W^s(\CC_2)$ and $W^u(\CC_2)$ meets transversely $W^s(\CC_1)$, giving rise to what we usually call \emph{heteroclinic tangle}. The \emph{Rank-one maps}'s theory  does not apply in this context.
 
\bigbreak
\begin{maintheorem} 
\label{tangency1}
Let $f_{(0,\mu_2)} \in\mathfrak{X}_{\Gamma}^4(\mathbb{V})$, with $\mu_2>0$. The flow of \eqref{general_R3} satisfies the following properties for $\mu_2>0$:
\medbreak
\begin{enumerate}
\item the set $\Lambda \left(\mathcal{G}_{(0, \mu_2)}\right)$ contains a horseshoe with infinitely many branches.
\medbreak
\item  there is a sequence $(\mu_{2,i})_{i\in \NN}$ of positive real numbers converging to zero, such that the manifolds $W^u(\CC_2)$ and $W^s(\CC_2)$ meet tangentially\footnote{This tangency is quadratic (generic). } for the flow of $f_{\left(0, \mu_{2,i}\right)}$.
\medbreak
\item  there exists a positive measure set  of parameters   in $I=[0,\varepsilon]$  so that  $\mathcal{G}_{(0, \mu_2)}$ admits a strange attractor with an ergodic SRB measure.
\medbreak
\item  there is a sequence $(\tilde{\mu}_{2,i})_{i\in \NN}$ of positive real numbers converging to zero, such that the flow of $f_{\left(0, \tilde\mu_{2,i}\right)}$ has a periodic sink.


\end{enumerate}
\end{maintheorem}
\bigbreak
\bigbreak
The proof of Theorem \ref{tangency1} does not depend on $\omega$ and follows the same lines to the reasoning of \cite{ ChOW2013, LR2017}.  
Item (1) of Theorem \ref{tangency1} is often called the classical  Birkhoff-Melnikov-Smale horseshoe theorem. 

\begin{remark}
Points in the horseshoe stated in Theorem \ref{tangency1}  lie on the topological closure of $W^u(\CC_2)\pitchfork W^s(\CC_1)$. \end{remark}

The horseshoe whose existence is proved in Theorem \ref{tangency1} has infinitely many saddle periodic points, whose Lyapunov mutipliers' modules tend to $+\infty$ and to $0$. The next notion will be useful in the sequel.

\begin{definition}
We say that the embedding  $\mathcal{G}_{(0, \mu_2)}: \Out(\CC_2) \to \Out(\CC_2)$ exhibits \emph{non-uniform expansion} if, given $\rho>0$, for Lebesgue almost all points in $ \Out(\CC_2)$, the map   $\mathcal{G}_{(0, \mu_2)}$ is  well defined and   has a positive upper Lyapunov exponent greater than $\rho$. 
\end{definition}
 
 The next result ensures the existence of a ``large'' strange attractor for the dynamics of \eqref{general_R3} and its non-uniform expansion for most parameters.  

\begin{mainproposition}\label{thm:G}
Let $f_{(0,\mu_2)} \in\mathfrak{X}_{\Gamma}^4(\mathbb{V})$, with $\mu_2>0$. If $\omega\gg1$, then the first return map $\mathcal{G}_{(0, \mu_2)} $ associated to \eqref{general_R3} exhibits a ``large'' strange attractor with \emph{non-uniform expansion}.  
\end{mainproposition}

The proof of this result is proved in Section \ref{Proof_ThE}.
In particular, it is possible to construct invariant probabilities  absolutely continuous with respect to the Lebesgue measure (cf. \cite[Sec. 3]{WY2006}).

\begin{maintheorem} 
\label{mixture}
Let $f_{(\mu_1,\mu_2)} \in\mathfrak{X}_{\Gamma}^4(\mathbb{V})$, with $\mu_1,\mu_2>0$. In the bifurcation parameter $(\mu_1, \mu_2)\in [0, \varepsilon]^2$, there exists a curve $Hom$ associated to the emergence of  homoclinic cycles to $\CC_1$.  Heteroclinic tangles occurs in the convex region defined by this curve. 
\end{maintheorem}
\bigbreak

\begin{figure}[ht]
\begin{center}
\includegraphics[height=5.5cm]{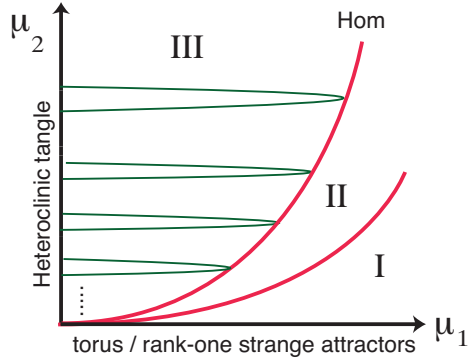}
\end{center}
\caption{\small Plausible bifurcation diagram associated to an element $f_{(\mu_1,\mu_2)}$ of the family $\mathfrak{X}_{\Gamma}^4(\mathbb{V})$.  I -- the flow has an invariant two-dimensional torus if $\omega\approx 0$ and a rank-one strange attractor if $\omega \gg 1$. II -- transition region;  III -- heteroclinic tangle; Hom -- curve that corresponds to the emergence of a homoclinic tangency associated to $\CC_1$. }
\label{bd1}
\end{figure}

The proof of Theorem \ref{mixture} is performed in Section \ref{Proof_ThF}. As suggested in Figure \ref{bd1}, for $\varepsilon>0$ small and $r\in [0,\varepsilon]$, defining $$B_r=\{(\mu_1, \mu_2)  \in [0, \varepsilon]\times [0, \varepsilon]: \mu_1^2+\mu_2^2\leq r^2 \},$$ we have:
$$\lim_{r \rightarrow 0}\frac{\text{Leb}_2(\{(\mu_1, \mu_2)\in [0, r]\times [0, r] : \mathcal{F}_{(\mu_1, \mu_2)} \text{  exhibits heteroclinic tangles}\}\cap B_r)}{\text{Leb}_2(B_r)} =1,$$
where $\emph{Leb}_2$ denotes the usual two-dimensional Lebesgue measure. This is why we say that non-uniform hyperbolicity is a \emph{prevalent phenomena} in the problem under consideration. 

\section{Preliminaries: strange attractors and SRB measures}
\label{s:definitions}
In this section, we gather a collection of technical facts used repeatedly in later sections.
We formalize the notion of strange attractor  for a two-parametric family of diffeomorphisms $H_{(a,b)}$ defined on  $M=     [0,1] \times \EU^1$,  endowed with the induced topology. The set $M$ is also called by \emph{circloid} in \cite{PPS}. In what follows, if $A\subset M$,   $\overline{A}$ denotes its topological closure.

\medbreak
Let $H_{(a,b)}$ be an embedding  such that $H_{(a,b)} (\overline{U} )\subset U$ for some open set $U\subset M$. In the present article, we refer to 
$$
{\Omega} =\bigcap_{m=0}^{+\infty}  H_{({a},b)}^m(\overline{U}).
$$
as an \emph{attractor} and $U$ its \emph{basin}. The attractor $\Omega$ is \emph{irreducible} if it cannot be written as the union of two (or more) disjoint attractors.
\medbreak
   
\begin{definition}
The embedding $H_{(a,b)}$ is said to have a \emph{horseshoe with infinitely many branches}  if %
there exists an invariant subset   $ \Sigma\subset  U$ on which $H_{(a,b)} |_\Sigma$ is topologically conjugated to a full shift of infinitely many symbols.
 \end{definition}

\begin{definition}
We say that $ H_{(a,b)}$ possesses a \emph{strange attractor supporting an ergodic SRB measure} $\nu$ if:
\begin{itemize}
\item for Lebesgue almost all $(y,\theta) \in U$, the $ H_{(a,b)}$--orbit of $(y,\theta)$ has a positive
Lyapunov exponent, \emph{ie}
$$
\lim_{n \in \NN} \frac{1}{n}\log  \|D  H_{({a},b)} ^n(y,\theta) \|>0;
$$

\item $ H_{({a},b)}$ admits a unique ergodic SRB measure (with no-zero Lyapunov exponents) \cite{WY};

\item  for Lebesgue almost all points $(y,\theta)\in U$ and  for every continuous test function $\varphi : U\rightarrow \RR$, we have:

\begin{equation}
\label{limit2}
\lim_{n\in \NN} \quad \frac{1}{n} \sum_{i=0}^{n-1} \varphi \circ  H_{({a},b)}^i(y,\theta) = \int  \varphi \, d\nu.
\end{equation}

\end{itemize}
\end{definition}

Admitting that $H_{({a},b)}$ admits a unique ergodic SRB measure $\nu$, we define convergence of  $H_{({a},b)}$ with respect to $\nu$ as follows:

\begin{definition}
We say that:
\begin{itemize}
\item $H_{({a},b)}$ converges  (in distribution with respect to $\nu$) to the normal distribution if, for every H\"older continuous function $\varphi: U \rightarrow \RR$, the sequence $\left\{\varphi\left( H_{( {a},b)}^i\right) : i\in \NN \right\}$ obeys a \emph{central limit theorem};  \emph{ie}, if $\dpt \int \varphi \, d \nu = 0$ then the sequence 
$
  \left(\frac{1}{\sqrt{m}} \dpt \sum_{i=0}^{m-1}\varphi \circ  H_{( {a},b)}^i\right)_m $ converges in distribution (with respect to $\nu$) to the \emph{normal distribution}.
 
\item the pair $(H_{({a},b)}, \nu)$ is \emph{mixing} if it is isomorphic to a Bernoulli shift. 

\end{itemize}
\end{definition}
We address the reader for \cite{WY} for more information on the subject. 

\section{Rank-one attractors' theory revisited}
\label{s: theory1}
To make a self-contained presentation, we provide an exposition of the theory of rank-one attractors adapted to our purposes.  We hope this saves the reader the trouble of going through the entire length of \cite{WO, WY} to achieve a complete description of the theory. 
\medbreak
  In what follows, let us denote by $C^3(\EU^1,\RR) $ the set of $C^3$--maps from $\EU^1$ (unit circle) to $\RR$. For $h\in  C^3(\EU^1,\RR) $, let 
$$C \equiv C(h)= \{\theta \in \EU^1 : h'(\theta) = 0\}$$ 
be the \emph{critical set} of $h$. For $\delta>0$, let $C_\delta$ be the $\delta$--neighbourhood of $C$ in $\EU^1$ and let $C_\delta$ be the $\delta$--neighbourhood of $\theta\in C$ as illustrated in Figure \ref{misiurewicz1}. The terminology \emph{dist} denotes the euclidian distance on $\RR$.

\subsection{Misiurewicz-type map}
\label{Misiurewicz-type map}
We say that $h\in  C^3(\EU^1,\RR) $ is a \emph{Misiurewicz-type map} (and we denote it by $h\in \mathcal{E}$) if the following assertions hold:
\bigbreak
\begin{enumerate}
\item There exists $\delta_0>0$ such that: \\
\begin{enumerate}
\item $\forall \theta \in C_{\delta_0}$, we have $h''(\theta)\neq 0$ and\\
\item $\forall \theta \in C$ and $n\in \NN$, $dist(h^n(\theta), C)\geq \delta_0$.\\
\end{enumerate}

\bigbreak

\item There exist constants $b_0, \lambda_0 \in \RR^+$ such that for all $\delta<\delta_0$ and $n\in \NN$, we may write: \\
\begin{enumerate}
\item if $h^k(\theta)\notin C_\delta$ for $k\in\{0, ,..., n-1\}$, then $|(h^n)'(\theta)|\geq b_0\,  \delta\,  \exp(\lambda_0\, n)$ and\\
\item  if $h^k(\theta)\notin C_\delta$ for $k\in\{0, ,..., n-1\}$ and $h^n(\theta)\in C_{\delta_0} $, then $|(h^n)'(\theta)|\geq b_0\,     \exp(\lambda_0\, n)$.\\
\end{enumerate}
\bigbreak
\end{enumerate}

Maps in $\mathcal{E}$ are among the simplest with non-uniform expansion. For $\delta_0>0$,  the set $C_{\delta_0}$ induces a partition on $\EU^1$, \emph{ie} the space $\EU^1$ may be divided in $C_{\delta_0}$ and $\EU^1\backslash C_{\delta_0}$.

\subsection*{Digestive remarks about Misiurewicz-type maps} 
\bigbreak
\begin{enumerate}
\item The critical orbits  stay a fixed distance away from the critical set $C$; \\
\item The derivatives grow at a uniform exponential rate (up to a prefactor) along orbits that remain outside $C_\delta$;\\
\item For $\theta \in C_\delta \backslash C$, although $|h'(\theta)|$ is small, the orbit of $\theta \in \EU^1$ does not return to $C_\delta$ again until its derivative has regained an ``amount'' of exponential growth.
\end{enumerate}

\begin{figure}[ht]
\begin{center}
\includegraphics[height=6.9cm]{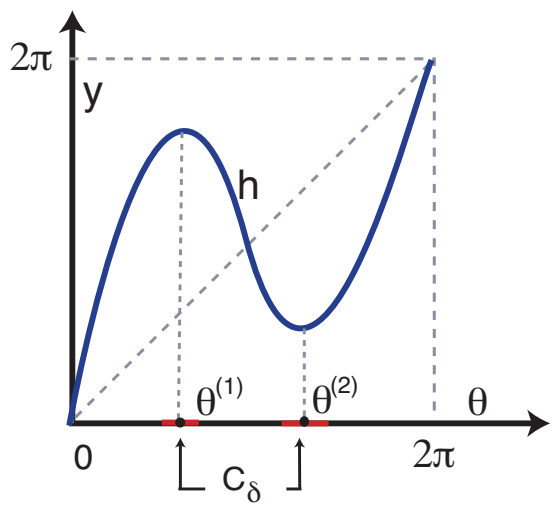}
\end{center}
\caption{\small  Example of a Misiurewicz-type map $h_a:\EU^1 \rightarrow \RR$. For $\delta>0$, the set  $C_{\delta}$  is a $\delta$-neighbourhood of the set of critical points $C$.    }
\label{misiurewicz1}
\end{figure}

\subsection{Admissible family}
We recall the notation and main results of \cite{WY2006}.
Let $$H : [0,2\pi] \times \EU^1 \rightarrow [0,2\pi]$$ be a $C^3$ map. The map $H$ defines a one-parameter family of maps
 $$\{h_a \in C^3([0,2\pi],[0,2\pi]) : a \in \EU^1\}$$ via $h_a(x) = H(x,a)$. We assume that there exists $a^\star \in \EU^1$ such that $h_{a^\star} \in \mathcal{E}$ (\emph{ie} $h_{a^\star}$  is a Misiurewicz map). For each $c \in C (h_{a^\star})$, there exists a continuation $c(a) \in C (h_a)$ provided $a$ is sufficiently close to $a^\star$. Therefore, for $a$ close to $a^\star$, let $C(h_{a^\star}) = \left\{c^{(1)}(a^\star), .... , c^{(q)}(a^\star)\right\}$, where
$$
\forall i \in \{1, ..., q-1\}, \qquad c^{(i)}(a^\star)<c^{(i+1)}(a^\star) \qquad \text{and} \qquad c^{(q+1)}(a^\star) = c^{(1)}(a^\star).
$$

From now on, when there is no risk of misunderstanding, we omit the dependence on $a^\star$ and the superscript $(i)$ in order to simplify the notation. 
For $c(a^\star) \in  C(h_{a^\star})$ we denote $$\beta_c(a^\star) = h_{a^\star} (c(a^\star)).$$ For all parameters $a$ sufficiently close to $a^\star$, there exists a unique continuation $\beta_c(a)$ of $\beta_c(a^\star)$ such that the orbits
$$
\{h^n_{a^\star}(\beta_c(a^\star)): n\in \NN\}\quad \text{and}\quad \{h^n_{a}(\beta_c(a)): n\in \NN\}
$$
have the same itineraries with respect to the partitions of $[0,1]$ induced by $C(h_{a^\star} )$ and $C(h_a)$, respectively. This means that:
$$
\forall n\in \NN, \qquad (h^n_{a^\star}(\beta_c(a^\star)) \in \left( c^{(j)} (a^\star),  c^{(j+1)} (a^\star))   \right)\quad  \Leftrightarrow \quad (h^n_{a}(\beta_c(a )) \in \left( c^{(j)} (a ),  c^{(j+1)} (a ))  \right),
$$
for $j\in \{1,...,q\}$.  In addition:
 
 \begin{lemma}[\cite{WY2006}] The map $a \mapsto \beta_c(a)$ is differentiable.
\end{lemma}

The previous lemma will be implicitly used in the next definition.

\begin{definition}
\label{def:admissible}
Let $H : [0,1] \times \EU^1 \rightarrow [0,1]$ be a $C^3$ map. The associated one-parameter family $\{h_a : a \in  \EU^1\}$ is \emph{admissible} if: \\ 
\begin{enumerate}
\item there exists $a^\star \in \EU^1$ such that $h_{a^\star} \in \mathcal{E}$; \\
\item for all $c\in C(h_{a^\star})$, we have
$$
\xi(c)=\frac{d}{da} (h_a(c(a)-\beta_c(a))\big|_{a=a^\star} = \frac{d}{da} (h_a(c(a)- h_{a^\star} (c(a^\star))\big|_{a=a^\star}  \neq 0.
$$
\end{enumerate}
\end{definition}

\subsection{Rank-one maps}
\label{rank_one}

\bigbreak

Let $M= [0,2\pi]\times \EU^1$, induced with the usual topology. We consider the two-parameter family of maps $H_{(a,b)}: M \rightarrow M$, where $a \in\EU^1$ and $b \in \RR$ is a scalar. Let $B_0 \subset \RR\backslash\{0\}$ with  0 as an accumulation point; this will be a crucial point in order to prove our results in Subsection \ref{ss: singular limit}. We assume the following conditions:
\bigbreak

\begin{description}
\item[\text{(H1) Regularity conditions}]  
\begin{enumerate}  
\item For each $b\in B_0$, the function $(x,y,a)\mapsto H_{(a,b)}$ is at least $C^3$--smooth.
\item Each map $H_{(a,b)}$ is an embedding of $M$ into itself.
\item There exists $k\in \RR^+$ independent of $a$ and $b$ such that for all $a \in  \EU^1$, $b\in B_0$ and $(y_1, \theta_1), (y_2, \theta_2) \in M$, we have:
$$
\frac{|\det DH_{(a,b)}(y_1, \theta_1)|}{|\det DH_{(a,b)}(y_2, \theta_2)|} \leq k.
$$
\end{enumerate}

\bigbreak
\item[\text{(H2) Existence of a singular limit}] 
For $a\in \EU^1$, there exists a map $$H_{(a,0)}: M \rightarrow  \{0\}\times \EU^1$$ such that the following property holds: for every $(y, \theta) \in M$ and $a\in [0,2\pi]$, we have
$$
\lim_{b \rightarrow 0} H_{(a,b)}(y, \theta) = H_{(a,0)}(y, \theta).
$$

\bigbreak
\item[\text{ (H3) $C^3$--convergence to the singular limit}]   For every choice of $a\in \EU^1$, the maps $(y, \theta, a) \mapsto H_{(a,b)}$  converge in the $C^3$--topology to $(y, \theta, a) \mapsto H_{(a,0)}$ on $M\times\EU^1$ as $b$ goes to zero.

\bigbreak
\item[\text{(H4) Existence of a sufficiently expanding map within the singular limit}]

The\-re exists $a^\star \in \EU^1$  such that $h_{a^\star}(\theta)\equiv H_{(a^\star, 0)}(0, \theta)$ is a Misiurewicz-type map.
\bigbreak
\item[\text{(H5) Parameter transversality}]  Let $  C_{a^\star}$ denote the critical set of a Misiu\-rewicz-type map $h_{a^\star}$. For each $x\in C_{a^\star}\equiv C(h_{a^\star})$, let $p = h_{a^\star}(x)$, and let $ {x(\tilde{a})}$ and $ {p(\tilde{a})}$ denote the 
continuations of $x$ and $p$, respectively, as the parameter $a$ varies around $a^\star$. The point $ {p(\tilde{a})}$  is the unique point such that $ {p(\tilde{a})}$  and $p$ have identical symbolic itineraries under $h_{a^\star}$ and $h_{\tilde{a}}$, respectively. We have:
$$
\frac{d}{da} h_{\tilde{a}}(x(\widetilde{a}))|_{a=a^\star} \neq \frac{d}{da} p(\tilde{a})|_{a=a^\star}.
$$

\bigbreak

\item[\text{(H6) Nondegeneracy at turns}] For each $x\in C_{a^\star}$, we have
$$
\frac{d}{dy} H_{(a^\star,0)}(y,\theta) |_{y=0} \neq 0.
$$

\bigbreak

\item[\text{(H7) Conditions for mixing}] If $J_1, \ldots, J_r$ are the intervals of monotonicity of a Misiu\-rewicz-type map $h_{a^\star}$, then:
\medbreak
\begin{enumerate}
\item $\exp(\lambda_0/3)>2$ (see the meaning of $\lambda_0$ in Subsection \ref{Misiurewicz-type map}) and
\medbreak
\item if $Q=(q_{im})$ is the matrix of all possible transitions between the intervals of monotonicity of $h_{a^\star}$ defined by:
\begin{equation*}
\left\{
\begin{array}{l}
1 \qquad \text{if} \qquad J_m\subset h_{a^\star} (J_i)\\  
0 \qquad \text{otherwise},\\
\end{array}
\right.
\end{equation*}
then there exists $N\in \NN$ such that $Q^N>0$ (in other words,  all entries of the matrix $Q^N$, endowed with the usual product topology,  are positive).
\bigbreak
\end{enumerate}

\end{description}

\begin{remark}
\label{terminologia1}
By \textbf{(H2)}, identifying $\EU^1 \times \{0\}$ with $\EU^1$, we refer to $H_{(a,0)}$  the restriction $h_a : \EU^1 \rightarrow \EU^1$ defined by $h_a(\theta) = H_{(a,0)}(\theta, 0)$ as the \emph{singular limit} of $H_{(a,b)}$.
\end{remark}

\subsection{Wang and Young's reduction}

 The results developed in \cite{WY, WY2003} are about maps with attracting sets on which there is strong dissipation
and (in most places) a single direction of instability.  Two-parameter families
${H_{(a,b)}}$ have been considered  and it has been proved that if a singular limit makes sense (for $b=0$) and if the resulting family of one-dimensional maps has certain
``good'' properties, then some of them can be passed back to the two-dimensional system ($b > 0$). They in
turn allow us to prove results on strange attractors for a positive Lebesgue measure set
of $a$. 

Conditions \textbf{(H1)--(H7)} are simple and checkable; when satisfied, they guarantee the existence of  strange attractors with a package of statistical and geometric properties:
\medbreak
\begin{theorem}[\cite{WY}, adapted]
\label{th_review}
Suppose the family $H_{(a,b)}$ satisfies \textbf{(H1)--(H7)}. Then, for all sufficiently small $b\in  B_0$, there exists a subset $\Delta \subset [0,2\pi]$ with positive Lebesgue measure such that for $a\in \Delta$, the map $H_{({a},b)}$ admits an irreducible strange attractor $\tilde{\Omega}\subset \Omega$  that supports a unique ergodic SRB measure $\nu$. The orbit of Lebesgue almost all points in $\tilde{\Omega} $ has positive Lyapunov exponent  and is asymptotically distributed according to $\nu$. 
\end{theorem}
\medbreak
The  map $H_{({a},b)}$ has \emph{exponential decay of correlations} for H\"older continuous observables. The theory may be extended for $  [0,1]^{N-1}\times \EU^1$, with $N\geq 2$ \cite{WY2006}.

\subsection{Periodic attractors in singular limits of families of rank-one maps}

We now introduce the combinatorics needed to   prove Theorem \ref{th_review2}.
Let $\delta<\delta_0$ be fixed ($\delta_0>0$ is the constant coming from the definition of \emph{Misiurewicz-type map} in Subsection \ref{Misiurewicz-type map}). 
For $ 1\leq i\leq q$, let $J^{(i)}$ be a subinterval of $C^{(i)}_\delta$, the connected component of $C_\delta$ containing the critical point $c^{(i)}$, and assume that there exist $n = n(i)$ and $j = j(i)$ associated to $J^{(i)}$ such that: \\

\begin{enumerate}
\item $h^k(J^{(i)})\cap C_\delta =\emptyset$ for all $0<k<n$ and \\
\item $h^n(J^{(i)})= C_\delta^{(j)}$.\\
\end{enumerate}

In other words, we have: \\

\begin{description}
\item[$q$] number of connected components of $\EU^1\backslash C_\delta$ ($q\geq 1$); \\
\item[$n(i)$] number of interactions needed to $J^{(i)}$ to intersect the critical set; \\ 
\item[$j(i)$] label of the connected component of the critical set intersected by $h^{n(i)}(J^{(i)})$.\\
\end{description}

\bigbreak
Now, for $\delta>0$ fixed, define the collection:

$$
J_\delta =\left\{  \left(J^{(i)}, n(i), j(i)\right): 1\leq j\leq q\right\}
$$
\bigbreak
We associate a \emph{directed graph} $\mathcal{P}(J_\delta )$ with $J_\delta $ as follows: \\
\begin{itemize}
\item the graph $\mathcal{P}(J_\delta )$  contains $q$ vertices $v_1,..., v_q$
 representing $c_1,..., c_q$;\\
 \item there exists a directed edge from $v_i$ to $v_\ell$ in $\mathcal{P}(J_\delta )$ if and only if $j(i) = \ell$.\\
\end{itemize}

\noindent According to \cite{OW2010}, we define the concept of completely accessible vertex.
\begin{definition}
 We say that a vertex $v_{i_0}$ in  $\mathcal{P}(J_\delta )$  is \emph{completely accessible} if for every $ 1\leq i\leq q$, there
exists a directed path from $v_i$ to $v_{i_0}$  in the graph $\mathcal{P}(J_\delta )$. 
\end{definition}

For fixed $\lambda<\lambda_0/5$ and   $\alpha>0$ small, let  $\Delta(\lambda, \alpha)$ be the set of $a\in \EU^1$ for which the following conditions hold for a critical point $c\in C\equiv C(h_a)$:\\
\begin{description}
\item[\textbf{(CE1)}]   $dist(h_a^n(c), C ) \geq \min\{\delta_0/2,  \exp (-\alpha n)\}$; \\
\item[\textbf{(CE2)}]    $|(h_a^n)'(h_a(c))|\geq 2b_0\delta_0 \exp (\lambda n)$.\\
\end{description}
These assertions are usually called by $(\lambda, \alpha)$--\emph{Collet-Eickmann conditions}. Next lemma says that, in the $C^3$--topology, if $a^\star \in \EU^1$ is such that $h_{a^\star}\in \mathcal{E}$, we may ``easily" find other values close to $a^\star$ for which $h_a\in \mathcal{E}$. \bigbreak
\begin{lemma}[\cite{WY2006}, adapted]
 If $a^\star\in \EU^1$ is such that $h_{a^\star}\in \mathcal{E}$ then 
 $$
 \liminf_{r\rightarrow 0^+}\frac{\emph{Leb}_1(\Delta(\lambda, \alpha)\cap [a^\star-r,a^\star+r])}{2r}>0,
 $$
 where $\emph{Leb}_1$ denotes the usual one-dimensional Lebesgue measure.
 \end{lemma}

 \bigbreak

\begin{theorem}[\cite{OW2010}, adapted]
\label{th_review2}
 
 Let $\{h_a : a \in \EU^1\}$ be an admissible family and let $a^\star \in \EU^1$ be such that $h_{a^\star} \in \mathcal{E}$. Fix $\lambda<\lambda_0/5$. Then for $\alpha < \lambda$ sufficiently small, there exists $\delta_1 > 0$ sufficiently small such that the following holds. If $h_{a^\star}$ admits a collection $J_\delta $ such that the directed graph $\mathcal{P}(J_\delta)$ has a completely accessible vertex for some $\delta<\delta_1$, then for every $\hat{a} \in  \Delta(\lambda, \alpha)$ sufficiently close to $a^\star$, there exists a sequence $a_n$ converging to $\hat{a}$ such that for every $n \in \NN$, the map $h_{a_n}$ admits a  periodic sink.

\end{theorem}
The periodic sink is \emph{superstable} because it is a critical point for the singular limit  ($\Rightarrow$ one of the Lyapunov multipliers is very close to $0$).

\section{Computation of the first return map}
\label{first return map}
For $\varepsilon>0$, in this section we denote by $B_{\varepsilon}(O_i)$ the set of points $X\in \mathcal{V}\subset \RR^2$ such that $\|X -O_i\|<\varepsilon$.
When the phase space of \eqref{general_R2_perturbed} is augmented with a $\EU^1$ factor, the hyperbolic saddles $O_1$ and $O_2$ of \eqref{general_R2}
become hyperbolic periodic solutions that we call by $\CC_1$ and $\CC_2$. These hyperbolic periodic orbits persist
for $\mu=(\mu_1, \mu_2)$ sufficiently small (in the $C^k$--norm, $k\geq 3$). 
By \cite{WO}, under hypotheses \textbf{(P1)--(P5)}, there exist  $\varepsilon_0 , \mu_0> 0$ and  a $\mu$-dependent coordinate system $(x, y, \theta)$ 
defined on the open set $V_i=B_{\varepsilon_0}(O_i)\times \EU^1$ such that for every $\mu =(\mu_1, \mu_2)\in  [0,\mu_0]\times  [0,\mu_0]$, we may write $$\CC_i =  \left\{\left(x_1^{(i)}, x_2^{(i)}, \theta^{(i)}\right): \quad x_1^{(i)}=x_2^{(i)} = 0, \quad \theta^{(i)} \in \EU^1\right\}$$ and the stable and unstable manifolds are locally flat:

$$
W^s(\CC_i) \cap V_i \subset \left\{\left(x_1^{(i)}, x_2^{(i)}, \theta^{(i)}\right): \quad x_2^{(i)}=0, \quad \theta^{(i)} \in \EU^1 \right\}
$$
and
$$
W^u( \CC_i) \cap V_i \subset \left\{\left(x_1^{(i)}, x_2^{(i)}, \theta^{(i)}\right): \quad x_1^{(i)}=0, \quad \theta^{(i)} \in \EU^1\right\}.
$$

For $\mu \in \,\, ]0, \mu_0]\times \, ]0, \mu_0]$  define the cross sections:

\begin{eqnarray*}
\In (\CC_1)&=& \left\{ \left(x_1^{(1)}, x_2^{(1)}, \theta^{(1)}\right):  \qquad x_1^{(1)}=\varepsilon_0, \quad -\|\mu\|/C_1 \leq x_2^{(1)}\leq C_1\|\mu\| , \quad \theta^{(1)} \in \EU^1\right\} \\ \\
\Out (\CC_1)&=& \left\{ \left(x_1^{(1)}, x_2^{(1)}, \theta^{(1)}\right):  \qquad x_2^{(1)}=\varepsilon_0, \quad 0 \leq x_1^{(1)}\leq C'_1\|\mu\| , \quad \theta^{(1)} \in \EU^1 \right\} \\ \\
\In (\CC_2)&=& \left\{ \left(x_1^{(2)}, x_2^{(2)}, \theta^{(2)}\right):  \qquad x_1^{(2)}=\varepsilon_0, \quad -\|\mu\|/C_2 \leq x_2^{(2)}\leq C_2\|\mu\| , \quad \theta^{(2)} \in \EU^1\right\} \\ \\
\Out (\CC_2)&=& \left\{\left(x_1^{(2)}, x_2^{(2)}, \theta^{(2)}\right): \qquad x_2^{(2)}=\varepsilon_0, \quad -\|\mu\|/C'_2 \leq x_1^{(2)}\leq C'_2\|\mu\|, \quad \theta^{(2)} \in \EU^1\right\}
\end{eqnarray*}
where the constants $C_i>0$ are suitably chosen and $C'_i$ satisfy $C'_i \mu_0 \ll \varepsilon_0$.

\subsection{Magnified coordinates}
\label{coords}
For $\mu =(\mu_1, \mu_2)\in \, \, ] 0, \mu_0]\times ]0, \mu_0]$, we make following change of coordinates:
\begin{equation}
\label{mag_coords1}
 \left(\|\mu\|\,  y_1^{(1)}, \|\mu\| \, y_2^{(1)}, \theta^{(1)}\right) \mapsto  \left(x_1^{(1)},\,   x_2^{(1)}, \theta^{(1)}\right)
\end{equation}
and 
\begin{equation}
\label{mag_coords2}
 \left(\|\mu\|\, y_1^{(2)}, \|\mu\|\, y_2^{(2)}, \theta^{(2)}\right) \mapsto  \left(  x_1^{(1)},\,  x_2^{(2)}, \theta^{(2)}\right)
\end{equation}
and therefore we obtain (see Table 2): 
\begin{eqnarray*}
\In (\CC_1)&=& \left\{ \left(y_1^{(1)}, y_2^{(1)}, \theta^{(1)}\right):  \qquad y_1^{(1)}=\varepsilon_0/\|\mu\|, \quad - 1/C_1 \leq y_2^{(1)}\leq C_1 , \quad \theta^{(1)} \in \EU^1\right\} \\ \\
\Out (\CC_1)&=& \left\{\left(y_1^{(1)}, y_2^{(1)}, \theta^{(1)}\right):   \qquad y_2^{(1)}=\varepsilon_0/\|\mu\|, \quad 0 \leq y_1^{(1)}\leq C'_1  , \quad \theta^{(1)} \in \EU^1 \right\} \\ \\
\In (\CC_2)&=& \left\{\left(y_1^{(1)}, y_2^{(1)}, \theta^{(1)}\right):   \qquad y_1^{(2)}=\varepsilon_0/\|\mu\|, \quad -1/C_2 \leq y_2^{(2)}\leq C_2\  , \quad \theta^{(2)} \in \EU^1\right\} \\ \\
\Out (\CC_2)&=& \left\{\left(y_1^{(1)}, y_2^{(1)}, \theta^{(1)}\right):  \qquad y_2^{(2)}=\varepsilon_0/\|\mu\|, \quad -1/C'_2 \leq y_1^{(2)}\leq C'_2 , \quad \theta^{(2)} \in \EU^1\right\}.
\end{eqnarray*}

\begin{table}[htb]
\begin{center}
\label{Table3}
\begin{tabular}{|c|c||c|c|} \hline 
{Set}  & \qquad  Notation \qquad  \qquad & {Set}  & \qquad  Notation \qquad  \qquad   \\
\hline \hline
&&&\\
$\In (\CC_1) $ &  $\left(y_1^{(1)}, y_2^{(1)}, \theta^{(1)}\right)$ & $\In (\CC_2) $ & $\left(y_1^{(2)}, y_2^{(2)}, \theta^{(2)}\right)$   \\  & && \\
 \hline \hline  &&&\\
$\Out (\CC_1) $ &$\left(\overline{y}_1^{(1)}, \overline{y}_2^{(1)}, \overline{\theta}^{(1)}\right)$ &$\Out (\CC_2) $ &$\left(\overline{y}_1^{(2)}, \overline{y}_2^{(2)}, \overline{\theta}^{(2)}\right)$   \\ 
&&&  \\ \hline

\hline

\end{tabular}
\end{center}
\label{notationA}
\bigskip
\caption{\small Coordinates and notation of the cross sections after the change of coordinates \eqref{mag_coords1} and  \eqref{mag_coords2}. }
\end{table} 

For $i\in \{1,2\}$, the sets $\In (\CC_i)$ may be divided as follows: $$
\In (\CC_i)\quad  =\quad  \In^+ (\CC_i) \quad \dot{\cup} \quad W^s(\CC_i) \quad \dot{\cup} \quad \In^- (\CC_i)  ,
$$
according to the sign of the $y_2^{(i)}$ coordinate: positive (negative) for initial conditions in $\In^+ (\CC_i)$ $(\In^- (\CC_i))$, zero for initial conditions in $W^s(\CC_i)$. We define analogously the sets $\Out^+ (\CC_2)$ and $\Out^- (\CC_2)$.

\bigbreak
For $i\in \{1,2\}$, we start by computing a normal form for \eqref{general_R3}  valid in $V_i\times   ]\, 0, \mu_0]^2$ and the associated local map near $\CC_i$, say $$Loc_i: \In (\CC_i)\backslash W^s(\CC_i) \rightarrow \Out(\CC_i).$$ The notation $\|\star\|_{\CC^3}$ denote the $C^3$--norm defined for maps defined in $V_i\times ]0, \mu_0]^2$. Recall that $V_i= B_{\varepsilon_0}(O_i) \times \EU^1$ is a genus two torus.
\begin{proposition}[\cite{WO}, adapted]
\label{local map prop}
System  \eqref{general_R3}  may be written, in terms of coordinates $\left(y_1^{(i)}, y_2^{(i)}, \theta^{(i)}\right)$ on $V_i\times [0, \mu_0]^2$,  in the following form:
\begin{equation}
\label{normal form 2}
\left\{
\begin{array}{l}
\dot y_1^{(i)}= \left(-c_i +\|\mu\| g_1\left(y_1^{(i)}, y_2^{(i)}, \theta^{(i)}; \mu \right) \right)y_1^{(i)}\\ \\
\dot y_2^{(i)}=\left(e_i +\|\mu\| g_2\left(y_1^{(i)}, y_2^{(i)}, \theta^{(i)}; \mu \right) \right) y_2^{(i)}\\ \\
\dot \theta^{(i)}= \omega \\
\end{array}
\right.
\end{equation}
There exists $K_1\in \RR^+$ such that the maps $g_1, g_2$ are analytic on $V_i\times ]\, 0, \mu_0]$ and satisfies $\| g_1\|_{\CC^3}, \|g_2 \|_{\CC^3}\leq K_1$, $i\in \{1,2\}$.
\end{proposition}

Notice that, by construction, $\mathbb{V}^\star$ is a neighbourhood of $\Gamma= \mathcal{L}_1\cup\mathcal{L}_2$ such that all solutions starting at $\mathbb{V}^\star$ remain inside $\mathbb{V}^\star$ for all positive $t$.

\bigbreak

For $\mu =(\mu_1, \mu_2)\in \, \, ] 0, \mu_0]^2$, let $q_0^{(i)}=\left(y_1^{(i)}(0), y_2^{(i)}(0), \theta^{(i)}(0)\right) \in \mathbb{V}^\star \cap \In (\CC_i)\backslash W^s(\CC_i)$ and let 
$$q^{(i)}\left(t, q_0^{(i)}; \mu\right)= \left[y_1^{(i)}\left(t, q_0^{(i)}; \mu\right),\, \,  y_2^{(i)}\left(t, q_0^{(i)}; \mu\right),\, \,  \theta^{(i)}\left(t, q_0^{(i)}; \mu\right)\right],\qquad t\geq 0$$ denote the unique solution of \eqref{normal form 2} with $q^{(i)}\left(0, q_0^{(i)}; \mu\right)=q_0^{(i)}$ (initial condition). Integrating \eqref{normal form 2}, we may write:

\begin{equation}
\label{normal form 3}
\left\{
\begin{array}{l}
 y_1^{(i)}\left(t, q_0^{(i)}; \mu\right)=  y_1^{(i)}(0) \exp  \dpt\int_0^t  \left[-c_i +\|\mu\| g_1^{(i)}\left(q^{(i)}\left(s, q_0^{(i)}; \mu\right)\right)ds\right]  \\ \\
 y_2^{(i)}\left(t, q_0^{(i)}; \mu\right)=  y_2^{(i)}(0)\exp \dpt \int_0^t \left[e_i +\|\mu\| g_2^{(i)}\left(q^{(i)}\left(s, q_0^{(i)}; \mu\right)\right)ds\right]  \\\\
\theta^{(i)}\left(t, q_0^{(i)}; \mu\right)= \theta_0^{(i)} + \omega t.\\
\end{array}
\right.
\end{equation}

The expression \eqref{normal form 3} may be rephrased as:

\begin{equation}
\label{normal form 4}
\left\{
\begin{array}{l}
y_1^{(i)}(t, q_0, \mu)= y_1^{(i)}(0) \exp \left( t  \left[-c_i + w_1^{(i)}\left(t, q_0^{(i)}; \mu\right)\right] \right)  \\ \\
 y_2^{(i)}(t, q_0, \mu)= y_2^{(i)}(0) \exp \left( t  \left[e_i + w_2^{(i)}\left(t, q_0^{(i)}; \mu\right)\right] \right)  \\ \\
\theta^{(i)}(t, q_0; \mu)= \theta_0^{(i)} + \omega t\\
\end{array}
\right.
\end{equation}
where  
\begin{equation}
\label{w_i expression}
w_j^{(i)}\left(t, q_0^{(i)}; \mu\right) = \frac{1}{t} \int_0^t \|\mu\|\, \,  g_j^{(i)}\left(q^{(i)}\left(s, q_0^{(i)}; \mu\right); \mu\right) \, ds, \qquad \text{for} \qquad j\in \{1,2\}. 
\end{equation}

The next result establishes the $C^3$--control of  $w_j$ on  $V_i  \times ]0, \mu_0]^2$, $i, j\in \{1,2\}$.

\begin{proposition}[\cite{WO}, adapted]
For $j\in {1,2}$, there exists $K_2\in \RR^+$ such that the following holds: for any $T^\star>1$ such that all solutions of \eqref{general_R3} that start in  $(\In(\CC_j)\backslash W^s(\CC_j))\cap \mathbb{V}^\star$  remain in $V_i$ up to time $T^\star$, we have $\| w_j\|_{\CC^3}\leq K_2\mu$. 
\end{proposition}

Let $q_0^{(i)} \in  \In(\CC_i) \backslash W^s(\CC_i)$ and   $\mu =(\mu_1, \mu_2)\in \, \, ] 0, \mu_0]^2$. The time of flight $T\equiv T\left(q_0^{(i)}; \mu \right)$ inside $V_i$ may be determined explicitly by solving the equation:
$$
\varepsilon_0/\|\mu\| = y_2^{(i)}(0) \exp \left(T \left(e_i + w_2^{(i)}\left(T, q_0^{(i)}; \mu\right)\right)\right),
$$
from where we deduce that
$$
T\equiv T \left(q_0^{(i)}; \mu\right)= \frac{1}{e_i +  w_2^{(i)}\left(T, q_0^{(i)}; \mu\right)}\ln \left(\frac{\varepsilon_0}{\|\mu\|\,  y_2^{(i)} (0)}\right).
$$
Proposition 5.7 and Lemma 7.4  of  \cite{WO}  provide a precise  control of $T$, in the $C^3$--norm (check the last paragraph of \cite[Sec. 7]{WO}).

\subsection{Local map}
\label{ss:local_map}
For $i\in \{1,2\}$, from now on, let us omit the constant components of the cross sections $\In^+(\CC_i)$ and $\Out(\CC_i)$. More specifically, let us use the covers (see Table 2):
$$
\In (\CC_i): \quad \left(y_1^{(i)}, y_2^{(i)}, \theta^{(i)}\right) \mapsto  \left( y_2^{(i)}, \theta^{(i)}\right)
$$
and $$
\Out (\CC_i): \quad \left(\overline{y}_1^{(i)}, \overline{y}_2^{(i)}, \overline\theta^{(i)}\right) \mapsto  \left(\overline{y}_1^{(i)},  \overline\theta^{(i)}\right). $$

 The local map $\emph{Loc}\, _i$ near $\CC_i$ sends $\left( y_2^{(i)}, \theta^{(i)}\right)\in \In^+(\CC_i)$ to  coordinates $ \left(\overline{y}_1^{(i)},  \overline\theta^{(i)}\right)$ in $\Out(\CC_i)$ and it is given by:
 \begin{eqnarray}
 \label{local_map1}
\overline{y}_1^{(i)} &=& \frac{\varepsilon_0}{\|\mu\|} \left[\frac{\varepsilon_0}{\|\mu\|\,  y_2^{(i)}} \right]^{-\delta_i} \\
\nonumber \overline\theta ^{(i)} &=& \theta^{(i)} + \frac{\omega}{e_i+w_2^{(i)}} \ln \left( \frac{\varepsilon_0}{\|\mu\|\, y_2^{(i)}}\right),
\end{eqnarray}
where 
\begin{equation}
\label{delta_i}
\delta_i \equiv \delta_i\left(t, q_0^{(i)}; \mu\right) ={\frac{c_i + w_1^{(i)}}{e_i +w_2^{(i)}}}>1.
\end{equation}
These formulas will be simplified later. 
 Note that $\dpt \lim_{(\mu_1, \mu_2)\rightarrow (0,0)} \delta_i\left(t, q_0^{(i)}; \mu\right) =c_i/e_i>1$.

A corresponding map can be constructed from $\In^-(\CC_i)$ to $\Out(\CC_i)$, but we are interested in trajectories following the  heteroclinic cycle $\Gamma$ in the positive $y_2^{(i)}$-direction. 

\subsection{The global map}
We assume that for $\mu=(\mu_1, \mu_2)\in\, ]0, \mu_0]^2$, the flow generated by \eqref{general_R3} induces a map from $\Out(\CC_i)$ into $\In(\CC_{i+1})$ satisfying  conditions \textbf{(P7a)} and \textbf{(P7b)} -- see Figure \ref{manifolds1}.
The global map $\Psi_{1\rightarrow 2}: \Out(\CC_1) \rightarrow \In(\CC_2)$ is given in the rescaled coordinates defined in Subsection \ref{coords}, by:

\begin{equation}
\label{global6.9}
\begin{array}{l}
y_2^{(2)}= \dpt b_1 \overline{y}_1^{(1)}  +\frac{\mu_1}{\|\mu\|} \phi_1\left( \overline{y}_1^{(1)}, \overline\theta^{(1)}\right)  \\ \\
\theta^{(2)}= \overline\theta^{(1)} + \xi_1 + \mu_1\psi_1\left(  \overline{y}_1^{(1)}, \overline\theta^{(1)}\right),
\end{array}
\end{equation}
where  $b_1 \neq 0$,  $ \phi_1\left( \overline{y}_1^{(1)}, \overline\theta^{(1)}\right)=\Phi_1\left(\|\mu\|\overline{y}_1^{(1)}, \overline\theta^{(1)}\right)$ and $\psi_1\left(  \overline{y}_1^{(1)}, \overline\theta^{(1)}\right)= \Psi_1\left( \|\mu\|\overline{y}_1^{(1)}, \overline\theta^{(1)}\right) $,  for $\Phi_1, \Psi_1$ are the maps defined in \textbf{(P7a)}.
Analogously, the global map $\Psi_{2\rightarrow 1}: \Out(\CC_2) \rightarrow \In(\CC_1)$ is given in the rescaled coordinates by:
\begin{equation}
\begin{array}{l}
y_2^{(1)}= b_2 \overline{y}_1^{(2)}  +\dfrac{\mu_2}{\|\mu\|}  \phi_2\left( \overline{y}_1^{(2)}, \overline\theta^{(2)}\right)  \\ \\
\theta^{(1)}= \overline\theta^{(2)} + \xi_2 + \mu_2\psi_2\left( \overline{y}_1^{(2)}, \overline\theta^{(2)}\right),
\end{array}
\end{equation}
for $b_2 \neq 0$, $ \phi_2\left( \overline{y}_1^{(2)}, \overline\theta^{(2)}\right)=\Phi_2\left(\|\mu\|\overline{y}_1^{(2)}, \overline\theta^{(2)}\right)$ and $\psi_2\left(  \overline{y}_1^{(2)}, \overline\theta^{(2)}\right)= \Psi_2\left( \|\mu\|\overline{y}_1^{(2)}, \overline\theta^{(2)}\right) $, where $\Phi_2, \Psi_2$ are the maps defined in \textbf{(P7b)}.

\begin{remark}
At this stage, we may need to slighly change the positive constants $C_1$, $C_2$, $C_1'$ and $C_2'$ in order that the global maps are well defined (this would correspond to ``shrink'' the domain of definition of  $\Psi_{1\rightarrow 2}$ and $\Psi_{2\rightarrow 1}$). We omit this technicality. 
\end{remark}

\begin{figure}[ht]
\begin{center}
\includegraphics[height=8.9cm]{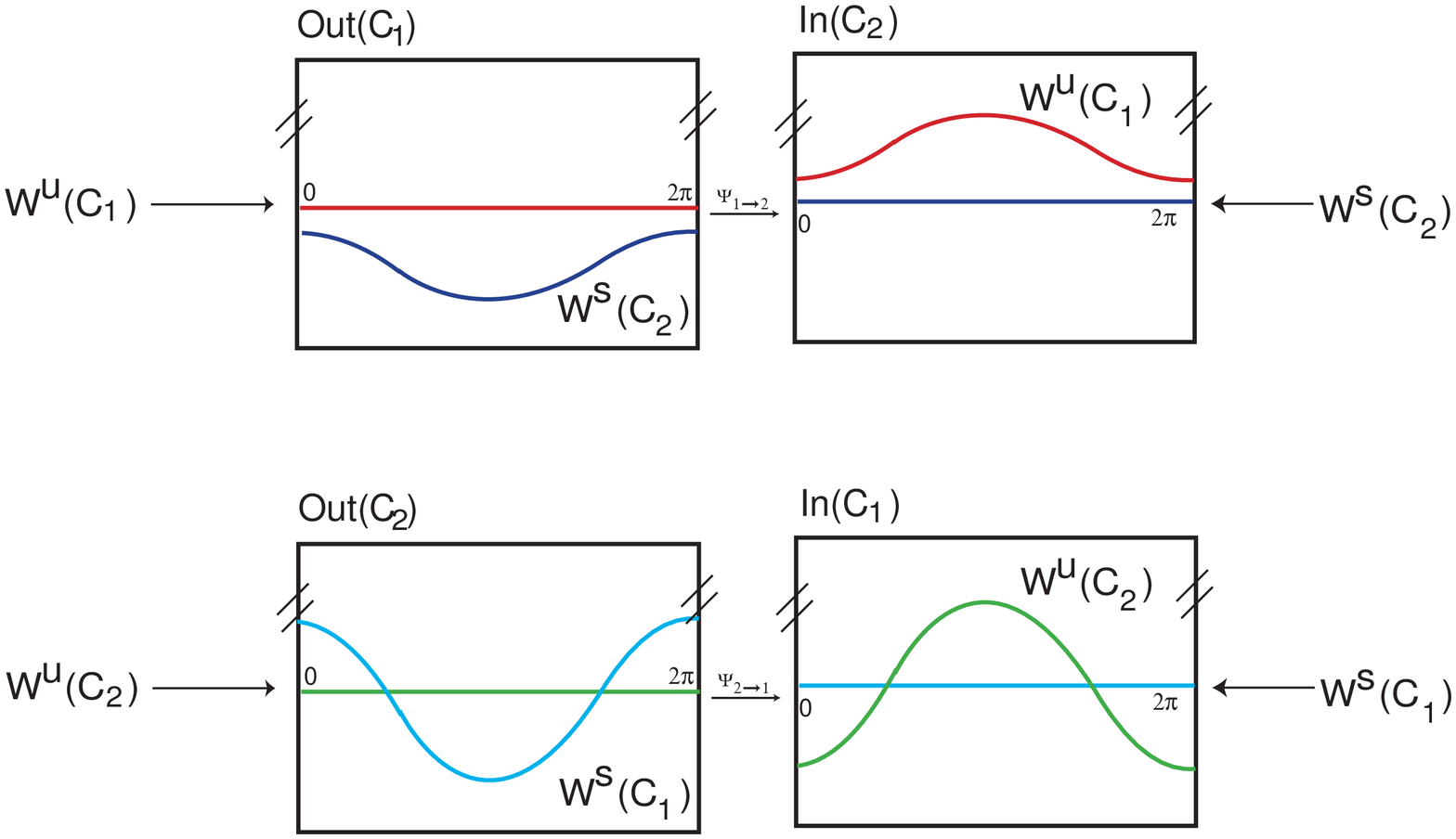}
\end{center}
\caption{\small Transition maps from $\Out(\CC_1)$ to $\In (\CC_2)$ (Case I) and from $\Out(\CC_2)$ to $\In (\CC_1)$ (Case II). In Case I: $W^u(\CC_1) \cap W^s(\CC_2) = \emptyset$. In Case II: $W^u(\CC_2) \pitchfork W^s(\CC_1)$.  Double bars mean that the sides are identified.  }
\label{manifolds1}
\end{figure}

\subsection{First return maps to a cross section}
\label{ss: 6.4}
For the flow of \eqref{general_R3}, the first return map to $\Out^+(\CC_1)$,  $$\mathcal{F}_{(\mu_1,0)}= \emph{Loc}\, _1 \circ \Psi_{2\rightarrow 1} \circ \emph{Loc}\, _2  \circ \Psi_{1\rightarrow 2}:\Out^+(\CC_1)\rightarrow \Out(\CC_1) $$ is then given by 

$$
\mathcal{F}_{(\mu_1,0)}\left(\overline{y}_1^{(1)} ,   \theta^{(1)}\right) = \left( \mathcal{F}_1  ,  \mathcal{F}_2\right) 
$$
where

\begin{eqnarray*}
    \mathcal{F}_1&=& {\varepsilon_0^{1-\delta_1}}\|\mu\|^{\delta_1-1}   \left[ b_2\varepsilon_0 ^{1-\delta_2}\| \mu \|^{\delta_2-1}   \left[ b_1\overline{y}_1^{(1)} + \frac{\mu_1}{\|\mu\|} \phi_1\left( \overline{y}_1^{(1)}, \overline\theta^{(1)}\right)\right]^{\delta_2}+ \frac{\mu_2}{\|\mu\|} \phi_2\left( \overline{y}_1^{(2)}, \overline\theta^{(2)}\right) \right]^{\delta_1}\\   \\ 
\mathcal{F}_2&=&  \overline\theta^{(1)}+ \xi_1+\xi_2 + \mu_1\psi_1\left(  \overline{y}_1^{(1)}, \overline\theta^{(1)}\right) +  \mu_2\psi_2\left( \overline{y}_1^{(2)}, \overline\theta^{(2)}\right) -
\left(\frac{\omega}{e_2+\omega_2^{(2)}} +\frac{\omega}{e_1+\omega_2^{(1)}}\right )\ln \left(\frac{\| \mu\|}{\varepsilon_0}\right) \\
&& - \frac{\omega}{e_2+\omega_2^{(2)}}\ln \left[ b_1\overline{y}_1^{(1)} + \frac{\mu_1}{\|\mu\|} \phi_1\left( \overline{y}_1^{(1)}, \overline\theta^{(1)}\right)\right] \\
&&-\frac{\omega}{e_1+\omega_2^{(1)}} \ln \left[b_2 \varepsilon_0^{1-\delta_2} \|\mu\|^{\delta_2-1}\left[ b_1\overline{y}_1^{(1)} + \frac{\mu_1}{\|\mu\|} \phi_1\left( \overline{y}_1^{(1)}, \overline\theta^{(1)}\right)\right]^{\delta_2}+ \frac{\mu_2}{\|\mu\|} \phi_2\left( \overline{y}_1^{(2)}, \overline\theta^{(2)}\right)\right].
\end{eqnarray*}

The first return map to $\Out(\CC_2)$,  $\mathcal{G}_{(0, \mu_2)}= \emph{Loc}\, _1 \circ \Psi_{1\rightarrow 2} \circ \emph{Loc}\, _1  \circ \Psi_{2\rightarrow 1} :\Out(\CC_2)\rightarrow \Out(\CC_2) $ is given by 

$$
\mathcal{G}_{(0, \mu_2)} \left(\overline{y}_1^{(2)} ,   \theta^{(2)}\right) = \left( \mathcal{G}_1  ,  \mathcal{G}_2\right) 
$$
where

\begin{eqnarray*}
    \mathcal{G}_1&=& \frac{\varepsilon_0}{\|\mu\|}  \left[b_1 \left(\frac{\|\mu\|}{\varepsilon_0}\right)^{\delta_1} \left[ b_2 \overline{y}_1^{(2)} + \frac{\mu_2}{\|\mu\|} \phi_2\left( \overline{y}_1^{(2)}, \overline\theta^{(2)}\right)\right]^{\delta_1}+ \frac{\mu_1}{\|\mu\|} \phi_1\left( \overline{y}_1^{(1)}, \overline\theta^{(1)}\right) \right]^{\delta_2}\\   \\ 
\mathcal{G}_2&=&  \overline\theta^{(2)}+ \xi_1+\xi_2 + \mu_2\psi_2\left(  \overline{y}_1^{(2)}, \overline\theta^{(2)}\right) +  \mu_1\psi_1\left( \overline{y}_1^{(1)}, \overline\theta^{(1)}\right) -
\left(\frac{\omega}{e_1+\omega_2^{(1)}} +\frac{\omega}{e_2+\omega_2^{(2)}}\right )\ln \left(\frac{\| \mu\|}{\varepsilon_0}\right) \\
&& - \frac{\omega}{e_1+\omega_2^{(1)}}\ln \left[ b_2 \overline{y}_1^{(2)} + \frac{\mu_2}{\|\mu\|} \phi_2\left( \overline{y}_1^{(2)}, \overline\theta^{(2)}\right)\right] \\
&&-\frac{\omega}{e_2+\omega_2^{(2)}} \ln \left[b_1 \varepsilon_0^{1-\delta_1} \|\mu\|^{\delta_1-1}\left[ b_2\overline{y}_1^{(2)} + \frac{\mu_2}{\|\mu\|} \phi_2\left( \overline{y}_1^{(2)}, \overline\theta^{(1)}\right)\right]^{\delta_1}+ \frac{\mu_1}{\|\mu\|} \phi_1\left( \overline{y}_1^{(1)}, \overline\theta^{(1)}\right)\right].
\end{eqnarray*}

\subsection{Simplified model for the return maps $\mathcal{F}_{(\mu_1,0)} $ and $\mathcal{G}_{(0, \mu_2)}$  }
\label{ss:simplified}
In what follows we assume, without loss of generality, that $b_1=b_2=\varepsilon_0=1$. This assumption simplifies the formulas and do not restrict the generality of the results. 

\begin{lemma}
\label{return1}
If  $b_1=b_2=\varepsilon_0=1$, then the first return map $\mathcal{F}_{(\mu_1,0)}$ to $\Out(\CC_1)$ may be written as   $\mathcal{F}_{(\mu_1,0)}=( \mathcal{F}_1,  \mathcal{F}_2)$ where:
\begin{eqnarray*}
    \mathcal{F}_1&=& \mu_1^{\delta-1} \left[ \overline{y}_1^{(1)} +   \phi_1\left( \overline{y}_1^{(1)}, \overline\theta^{(1)}\right)\right]^{\delta}\\
\mathcal{F}_2&=&  \overline\theta^{(1)}+ \xi  + \mu_1\psi_1\left(  \overline{y}_1^{(1)}, \overline\theta^{(1)}\right) -
\omega K_F \ln \left({ \mu_1} \right)   - \omega K_F  \ln \left[ \overline{y}_1^{(1)} +  \phi_1\left( \overline{y}_1^{(1)}, \overline\theta^{(1)}\right)\right] \\
\end{eqnarray*}
with $\delta=\delta_1\delta_2$ (\footnote{See equation   \eqref{delta_i} for the definitions of $\delta_1$ and $\delta_2$.}), $\xi=\xi_1+\xi_2$ and $ K_F= \frac{1}{e_2+\omega_2^{(2)}} +\frac{ \delta_2}{e_1+\omega_2^{(1)}},$ where $\omega_2^{(1)}$ and $\omega_2^{(2)}$ are the integrals defined in \eqref{normal form 4}.
\end{lemma}

\begin{proof}
First of all note that when $\mu_1>0$ and $\mu_2=0$, then $\|\mu\|=\mu_1$. 
Using the results of \S \ref{ss: 6.4}, the first return map $\mathcal{F}_{(\mu_1, 0)}= \emph{Loc}_1 \circ \Psi_{2\rightarrow 1} \circ \emph{Loc}_2  \circ \Psi_{1\rightarrow 2} $ is given by 

$$
\mathcal{F}_{(\mu_1, 0)}\left(\overline{y}_1^{(1)} ,  \theta^{(1)}\right) = \left( \mathcal{F}_1  , \mathcal{F}_2   \right) 
$$
where
\begin{eqnarray*}
    \mathcal{F}_1&=& \mu_1^{\delta-1} \left[ \overline{y}_1^{(1)} +   \phi_1\left( \overline{y}_1^{(1)}, \overline\theta^{(1)}\right)\right]^{\delta}.\\
\end{eqnarray*}

Simplifying expression of $\mathcal{F}_2$ given in \S \ref{ss: 6.4}, we get:
\begin{eqnarray*}
\mathcal{F}_2&=&  \overline\theta^{(1)}+ \xi_1+\xi_2 + \mu_1\psi_1\left(  \overline{y}_1^{(1)}, \overline\theta^{(1)}\right) -
\left(\frac{\omega}{e_2+\omega_2^{(2)}} +\frac{\omega}{e_1+\omega_2^{(1)}}\right )\ln \left({ \mu_1} \right) \\
&& - \frac{\omega}{e_2+\omega_2^{(2)}}\ln \left[ \overline{y}_1^{(1)} +  \phi_1\left( \overline{y}_1^{(1)}, \overline\theta^{(1)}\right)\right] \\
&&-\frac{\omega}{e_1+\omega_2^{(1)}} \ln \left[\mu_1^{\delta_2-1}\left[ \overline{y}_1^{(1)} +\phi_1\left( \overline{y}_1^{(1)}, \overline\theta^{(1)}\right)\right]^{\delta_2} \right]\\
 &=&  \overline\theta^{(1)}+ \xi_1+\xi_2 + \mu_1\psi_1\left(  \overline{y}_1^{(1)}, \overline\theta^{(1)}\right) -
\left(\frac{\omega}{e_2+\omega_2^{(2)}} +\frac{\omega \delta_2}{e_1+\omega_2^{(1)}}\right )\ln \left({ \mu_1} \right) \\
&& - \left(\frac{\omega}{e_2+\omega_2^{(2)}}   +\frac{\omega\delta_2}{e_1+\omega_2^{(2)}} \right)\ln \left[ y_1^{(1)} +  \phi_1\left( \overline{y}_1^{(1)}, \overline\theta^{(1)}\right)\right] \\
 &=&  \overline\theta^{(1)}+ \xi  + \mu_1\psi_1\left(  \overline{y}_1^{(1)}, \overline\theta^{(1)}\right) -
\omega K_F  \ln \left({ \mu_1} \right)   - \omega K_F  \ln \left[ y_1^{(1)} +  \phi_1\left( \overline{y}_1^{(1)}, \overline\theta^{(1)}\right)\right] \\
\end{eqnarray*}
where $\xi $ and $K_F$ are as stated. 
\end{proof}

\begin{lemma}
\label{return2}
If  $b_1=b_2=\varepsilon_0=1$, then the first return map $\mathcal{G}_{(0,\mu_2)}$ to $\Out(\CC_2)\backslash W^s(C_2)$ may be written as   $\mathcal{G}_{(0,\mu_2)}=( \mathcal{G}_1,  \mathcal{G}_2)$ where:
\begin{eqnarray*}
    \mathcal{G}_1&=& \mu_2^{\delta-1} \left[ \overline{y}_1^{(2)} +   \phi_2\left( \overline{y}_1^{(2)}, \overline\theta^{(2)}\right)\right]^{\delta}\\
\mathcal{G}_2&=&  \overline\theta^{(2)}+ \xi  + \mu_2\psi_2\left(  \overline{y}_1^{(2)}, \overline\theta^{(2)}\right) -
\omega K_G \ln \left({ \mu_2} \right)   - \omega K_G  \ln \left[ \overline{y}_1^{(2)} +  \phi_2\left( \overline{y}_1^{(2)}, \overline\theta^{(2)}\right)\right] \\
\end{eqnarray*}
where $\xi=\xi_1+\xi_2$ and $K_G = \frac{1}{e_1+\omega_2^{(1)}} +\frac{ \delta_1}{e_2+\omega_2^{(2)}}.$

\end{lemma}
We omit the proof of Lemma \ref{return2} since it is similar to that of Lemma \ref{return1}.

\begin{remark}
Using  the integrals defined in \eqref{w_i expression}, note that when $\mu_1=\mu_2=0$, the formulas $\omega_2^{(1)} = \omega_2^{(2)} \equiv 0$. For this case, we obtain formulas \eqref{formulas3.1}:
$$
 K_F= \frac{1}{e_2 } +\frac{ \delta_2}{e_1}= \frac{e_1+c_2}{e_1e_2}\neq 0 \qquad \text{and} \qquad K_G=  \frac{1}{e_1} +\frac{ \delta_1}{e_2}=\frac{e_2+c_1}{e_1e_2} \neq 0 .
$$

\end{remark}

In order to improve the readability of the manuscript, we use the   the  terminology of Table 3 in the  Sections \ref{s: Proof Theorems A and B}--\ref{Proof_ThE}: 

\begin{table}[htb]
\begin{center}
\begin{tabular}{|c|c|} \hline 
&\\
$\mathcal{F}_{(\mu_1,0)} \mapsto \mathcal{F}_{\mu}$ & \quad Sections \ref{s: Proof Theorems A and B} and \ref{Proof_ThC} \qquad \\
&\\
\hline \hline
&\\
$\mathcal{G}_{(0,\mu_2)}\mapsto \mathcal{G}_{\mu}$ & \quad Sections \ref{s: proof Th D} and \ref{Proof_ThE} \qquad \\  &\\
 \hline \hline 
 \end{tabular}
\end{center}
\label{notationB}
\bigskip
\caption{\small Notation for the next sections where $\mu_1, \mu_2\in\, \,  ]0, \varepsilon]$.}
\end{table}

\section{Proofs of Theorems \ref{thm:B} and \ref{thm:F} }
\label{s: Proof Theorems A and B}
This case reports the scenario described by $\mu_1>0$ and  $\mu_2=0$. This is why we use the first return map $\mathcal{F}_{(\mu_1, 0)} \mapsto \mathcal{F}_\mu $ (cf. Table 3).

\subsection{The singular limit}
\label{ss: singular limit}
Let  $k: \RR^+ \rightarrow \RR$ be the invertible map defined by $$k(x)= -\omega K_F  \ln (x).$$
For $\mu_0<\varepsilon$, define now the decreasing sequence $(\mu_{n})_n$ such that, for all $n\in \NN$, we have: \\
\begin{enumerate}
\item  $\mu_{n} \in\, ]0,  \mu_0[ $ and \\
\item $k(\mu_{n}) \equiv 0 \mod 2\pi$.
\end{enumerate}
\bigbreak

\noindent Since $k$ is an invertible map,  for $a \in \EU^1  $ fixed and $n\geq n_0\in \NN$, let  
\begin{equation}
\label{sequence1}
\mu_{(a, n)}= k^{-1}(k(\mu_{n})+a)\,\,  \in \,\, ]0, \mu_0[.
\end{equation}
It is easy to check that: 
\begin{equation}
\label{sequence2}
k(\mu_{(a, n)})= -\omega K_F  \ln (\mu_{n} )+a=a \mod 2\pi.
\end{equation}
Define $\mathcal{F}_{(a, \mu_{(a, n)})}$ as $\mathcal{F}_{\mu_{(a, n)}}$. 
The following proposition establishes $\CC^3$--convergence to a singular limit as $n \rightarrow +\infty$.

\begin{lemma} 
\label{important lemma}
In the $C^3$--norm, for $a\in \EU^1$, the following equality holds:
$$
\lim_{n\in \NN} \|\mathcal{F}_{(a, \mu_{(a, n)})}   -(\textbf{0}, h_a)\| =0$$
where $\textbf{0}$ is the null map and 
\begin{equation}
\label{circle map}
h_a \left(  \overline\theta^{(1)}\right)=  \overline\theta^{(1)}+ \xi   +
a   - \omega K_F  \ln \left[    \phi_1\left( 0, \overline\theta^{(1)}\right)\right] . 
\end{equation}
\end{lemma}
\begin{proof}
The proof of this lemma follows from the fact  that $k(\mu_{(a, n)})=a \mod{2\pi}$. The $C^3$--convergence is a consequence of Hypothesis \textbf{(P2)}, Proposition \ref{local map prop} and Lemma \ref{return1}.  See also Lemma 7.4  of  \cite{WO}.
 \end{proof}

\begin{remark}
\label{rem7.2}The map $  h_a\left(\theta^{(1)}\right) =\theta^{(1)}+\xi+a  -\omega K_F \ln  \left(\phi_1(\theta^{(1)})\right)\equiv  \mathcal{F}_{(a, \mu_{(a, n)})}^1 \left(0, \theta^{(1)}\right)$ is a Morse function with finitely  many nondegenerate critical points (by Hypothesis \textbf{(P7a)}).
\end{remark}

\subsection{Verification of the hypotheses of the theory of rank-one maps.}
\label{ss:verification}
From now on, our focus will be the sequence of two-dimensional maps 
\begin{equation}
\label{family}
\mathcal{F}_{(a,b)}=  \mathcal{F}_{(a, \mu_{(a, n)})}\qquad \text{with} \qquad n \in \mathbb{N} \qquad \text{and} \qquad a \in \EU^1 \text{ fixed}.
\end{equation}
 Since our starting point  is an attracting heteroclinic cycle  (for $\mu_1=\mu_2=0$), the \emph{absorbing sets} defined in Subsection 2.4 of \cite{WY} follow from the existence of the attracting annular region. Now, we show that the family of maps \eqref{family}  satisfies Hypotheses \textbf{(H1)--(H6)} stated in Subsection \ref{rank_one}.

\medbreak
\begin{description}
\item[\text{(H1)}]  
The first two items are immediate.  We establish the distortion bound \textbf{(H1)(3)} by studying $D\mathcal{F}_{(a, \mu_{(n,a)})}$. Direct computation  implies that for every $\mu \in (0, \tilde\mu)$ and $\left( \overline{y}_1^{(1)}, \overline\theta^{(1)}\right) \in \Out(\CC_1)\cap \mathbb{V}^\star$, one gets:

 $$|\det D \mathcal{F}_{(a, \mu(n,a))}\left( \overline{y}_1^{(1)}, \overline\theta^{(1)}\right)| = |\det \emph{Loc}\, _1|| \det  \Psi_{2\rightarrow 1}|| \det  \emph{Loc}\, _2 ||\det  \Psi_{1\rightarrow 2} |$$
where
\begin{eqnarray*}
\left| \det  \Psi_{1\rightarrow 2} \left( \overline{y}_1^{(1)}, \overline\theta^{(1)}\right)\right|& = & \left|
\left(b_1 +\frac{\partial \phi_1}{\partial y}\right)\left(1+\mu_1 \frac{\partial \Psi_1 }{\partial \theta^{(1)}}\right)
 -\mu_1 \frac{\partial \psi_1}{\partial \overline{y}^{(1)}} \frac{\partial \phi_1}{\partial \overline\theta^{(1)}}\right| \\
\left| \det \emph{Loc}\, _2 \left(  {y}_2^{(2)},  \theta^{(2)}\right)\right|& = & \left| \mu_1^{\delta_2-1} \left(y_2^{(2)}\right)^{\delta_2-1} \right| \\
\left| \det  \Psi_{2\rightarrow 1} \left( \overline{y}_1^{(2)}, \overline\theta^{(2)}\right)\right|& = & \left|\left(b_2 +\frac{\partial \phi_2}{\partial y}\right)\left(1+\mu_1 \frac{\partial \Psi_2 }{\partial \theta^{(1)}}\right) -\mu_1 \frac{\partial \psi_2}{\partial \overline{y}^{(2)}} \frac{\partial \phi_2}{\partial \overline\theta^{(2)}}\right| \\
\left| \det \emph{Loc}\, _1 \left(  {y}_2^{(1)},  \theta^{(1)}\right)\right|& = & \left| \mu_1^{\delta_1-1} \left(y_2^{(1)}\right)^{\delta_1-1} \right| \\
  \end{eqnarray*}

Since  $\left(y_2^{(2)}\right)^{\delta_2-1}, \left(y_2^{(1)}\right)^{\delta_1-1}$ are positive and $b_1, b_2\neq 0$ (because $c_1, c_2\neq 0$ in \textbf{(P6)})  we conclude that there exists $\mu^\star>0$ small enough such that:
$$
\forall \mu \in \, ]\, 0, \mu^\star \, [, \qquad \left| \det D \mathcal{F}_{(a, \mu_{(n,a)})}\left( \overline{y}_1^{(1)}, \overline\theta^{(1)}\right)\right|  \in  \,\,  ]\, k_1^{-1}, k_1\, [ ,
$$
for some $k_1>1$. This implies that hypothesis \textbf{(H1)(3)}  is satisfied.
\bigbreak
\item[\text{(H2) and (H3)}] It follows from Lemma \ref{important lemma} where $b=\mu_{(n,a)}$ (see \eqref{family}). 

\bigbreak
\item[\text{(H4) and (H5)}] 
These hypotheses are connected with the family of circle maps $$h_a: \EU^1 \rightarrow \EU^1$$ defined in Remark \ref{rem7.2}. We now use the   following result:

\begin{proposition}[\cite{WY2003}, adapted]
\label{Prop 7.3}
Let  $\Phi: \EU^1 \rightarrow \RR$ be a $C^3$ function with nondegenerate critical points. Then there exist $L_1$ and $\delta$ such that
of $L\geq L_1$ and $\Psi: \EU^1 \rightarrow \RR$  is a $\CC^3$ map with $\|\Psi\|_{C^2}\leq \delta$ and $\|\Psi\|_{C^3}\leq 1$, then the family
$$
h_a(\theta)=\theta+a+L(\Phi(\theta) +\Psi(\theta)),\qquad a\in \EU^1
$$
satisfies \textbf{(H4) and (H5)}. If $L$ is sufficiently large, then    \textbf{(H7)} is also verified. 
\end{proposition}

\noindent  It is immediate to check that the family $h_a$
satisfies Properties \textbf{(H4)} and \textbf{(H5)}.

\bigbreak

\item[\text{(H6)}] The computation follows  from direct computation using  the expression of $\mathcal{F}_\mu\left(\overline{y}_1^{(1)},\overline{\theta}^{(1)}\right)$. Indeed, for each $\theta\in C_{a^\star}$ (set of critical points of $h_{a^\star}$ defined in \eqref{circle map}), we have
$$
\frac{d}{dy} \mathcal{F}_{(a,\mu_{(n,a)})}\left(\overline{y}_1^{(1)},\theta^{(1)} \right) \big|_{\overline{y}_1^{(1)}=0} \neq 0. $$

\bigbreak

\item[\text{(H7)}] It follows from Proposition \ref{Prop 7.3} if $\omega$ is large enough.

\end{description}

\bigbreak
We apply the theory developed by \cite{WY} to prove Theorems \ref{thm:B} and   \ref{thm:F}.

\subsection{Proof of Theorem \ref{thm:B}: attracting torus}
\label{Proof_ThA}
The map $h_a(\theta)=\theta+ \xi   +
a   - \omega K_F  \ln \left[    \phi_1\left(\theta\right)\right] $ is a diffeomorphism on the circle   if and only if:
 
\begin{eqnarray*}
h_a'(\theta)>0 &\Leftrightarrow&  1- \omega K_F  \frac{\phi_1'(\theta)}{\phi_1(\theta)}>0 \\ \\
&\overset{\eqref{global6.9}}{\Leftrightarrow}  &  1- \omega K_F  \frac{\Phi_1'(\theta)}{\Phi_1(\theta)}>0 \\ \\
 &\Leftrightarrow&  \omega\times  \sup_{\theta\in \EU^1}\frac{\Phi_1'(\theta)}{\Phi_1(\theta)}<1/ K_F .
\end{eqnarray*}
In particular, if $\dpt \omega\times  \sup_{\theta\in \EU^1}\frac{\Phi_1'(\theta)}{\Phi_1(\theta)}<\frac{1}{ K_F}$, the map $h_a$ is a diffeomorphism on $\mathcal{C}$  ($\Rightarrow$ the flow of \eqref{general_R3} has an invariant torus). The circle $\mathcal{C}$  is attracting by Lemma \ref{important lemma} and it is not contractible because it may be seen as the graph of a map. Theorem \ref{thm:B} is proved.
 
\subsection{Proof of  Theorem \ref{thm:F}: rank-one strange attractors}
\label{Proof_ThB}
 Since the family $\mathcal{F}_{(a, \mu(n,a))}$ satisfies \textbf{(H1)--(H7)} then, for $\mu^+=\min\{\varepsilon, \mu^\star\}>0$ and $\omega\gg 1$,  there exists a subset of $\Delta \subset [0, \mu^+]$ with positive Lebesgue
measure such that for $\mu\in \Delta$, the map $\mathcal{F}_{\mu}$ admits a strange attractor in 
$$ \Omega \subset \bigcap_{m=0}^{+\infty}  \mathcal{F}_{ \mu}^m(\Out^+(\CC_1))$$   supporting a unique ergodic SRB measure $\nu$. Denoting by $\text{Leb}_1$ the one-dimensional Lebesgue measure, from  the reasoning of   \cite[Sec. 3]{WO}, we have:  
\begin{equation}
\label{abunda2}
\liminf_{r\rightarrow 0^+}\, \,  \frac{ {Leb}_1 \left\{\mu \in [0,r]\cap \Delta: \mathcal{F}_\mu  \text{  has a strange attractor with a SRB measure}\right\}}{r}  >0.
\end{equation}
 Theorem \ref{thm:F} is shown.

 \subsection*{Technical remarks}

 \begin{enumerate}
 \item Throughout the proof, it is essential that the domain of definition of $\mathcal{F}_\mu$ is diffeomorphic to a cylinder. Otherwise the results of \cite{WO} cannot be applied. \\

\item   The SRB measure $\nu$ obtained in this result  is global in the sense that almost every point in $\Omega$ is generic with respect to $\nu$.
The orbit of Lebesgue almost all points in ${\Omega}$ has positive Lyapunov exponent  and is asymptotically distributed according to $\nu$.    \\

\item
 The strange attractor $\Omega$ is non-uniformly hyperbolic, non-structurally stable and is the limit of an increasing sequence of uniformly hyperbolic invariant sets.
\end{enumerate}

\section{Proof of  Theorem \ref{thm:C}}
\label{Proof_ThC}
In order to prove  Theorem \ref{thm:C}, we  use  Theorem \ref{th_review2}.
We make use of the fact that the family $h_a: \EU^1\rightarrow \EU^1$, $a\in \EU^1$, is admissible (see \S \ref{ss:verification}). 
In particular, there exists $a^\star \in \EU^1$ such that $h_{a^\star} \in \mathcal{E}$ (is a Misiurewicz map).
In this section we assume the following technical hypothesis on $\lambda_0>0$ (this constant comes from the definition of  Misiurewicz-type map of  \S \ref{def:admissible}). 
\bigbreak
 \begin{description}
\item[(TH)] $\exp(\lambda_0)>{2}$.
\end{description}

\bigbreak

Let $S_i$ be one of the connected components of $C_\delta^{(i)}\backslash\{c^{(i)}\}$, $i=1,..., q$. This interval is contained  in one monotonicity interval of $h_{a^\star}$ (see \textbf{(H7)}).
\begin{lemma}
\label{lemma1}
Let ${a^\star}\in \EU^1$ be such that $h_{a^\star}\in \mathcal{E}$. Then, there exists $m_1 \in \NN$ such that the following conditions hold:\\
\begin{enumerate}
\item $h_{a^\star}^{k} (S_i)\cap C_\delta = \emptyset$, for all $k\in \{1, ..., m_1-1\}$ and\\
\item $h_{a^\star}^{m_1} (S_i)\cap C_\delta \neq \emptyset$.
\end{enumerate}
\end{lemma}

\begin{proof}
Suppose, by contradiction, that $h_{a^\star}^{n}(S_i) $ never intersect $C_\delta$, for all $n\in \NN$. 
For $C\subset \EU^1$, let us denote by \emph{diam(C)} the maximum distance of points in $C$. Therefore,   for $y \in h_{a^\star}(S_i)$ we have:
\begin{eqnarray*}
\emph{diam}(h_{a^\star}^n(S_i)) &\geq&  \emph{diam}(h_{a^\star}(S_i)) \, m \qquad \text{where} \qquad m= \inf_{y\in h_{a^\star}(S_i)}\emph{diam}((h_{a^\star}^{n-1})'(y)) \\ \\ &>&  \emph{diam}(h_{a^\star}(S_i)) b_0 \delta \exp (\lambda_0(n-1)) \\ \\
  &>&  C   b_0 \delta \exp (\lambda_0(n-1)), \qquad \text{for some}\quad  C>0.
\end{eqnarray*}
Since $\dpt \lim_{n\rightarrow +\infty}C    b_0 \delta \exp (\lambda_0(n-1))=+\infty$ and $\emph{diam}( S_i) \leq 2\pi$, this is a contradiction. Then the images of $S_i$ under $h_{a^\star}$ should intersect $C_\delta$.
\end{proof}

From Lemma \ref{lemma1}, we may identify three disjoint possibilities: \\
\begin{enumerate}
\item there exists $j_0 \in \{1, ..., q\}$ such that $C_\delta^{(j_0)}\subset h_{a^\star}^{m_1}(S_i)$ $\Rightarrow n(i)=m_1$.\\
\item $]c^{(l)}, c^{(l+1)}[\subset h_{a^\star}^{m_1}(S_i)$ for some $l\in\{1, ..., q\}$.  Since $[0,2\pi]\subset h_{a^\star}(]c^{(l)}, c^{(l+1)}[)$ for all $l\in \{1, ..., q\}$ and $\omega \gg 1$ (remind that for $\omega\gg 1$ the map $h_{a^\star}$ is mixing by Proposition \ref{Prop 7.3}), it follows that  $h_{a^\star}^{m_1}(S_i)\cap C_\delta^{(j_0)}\neq \emptyset$ $\Rightarrow n(i)=m_1+1$.\\
\item none of the above.\\
\end{enumerate}

Under the notation of Lemma \ref{lemma1}, let $L_0$ be one  connected component of $h_{a^\star}^{m_1}(S_i)\backslash C_\delta$ with one endpoint at $h_{a^\star}^{m_1}(c^{(i)})$, $i=1,..., q$.

\begin{lemma}
\label{lemma2}
If \textbf{(TH)} holds, there exists $m_2\in \NN$ and a subinterval $L_1$ of $L_0$ such that $h_{a^\star}^k(L_1)\cap C_\delta=\emptyset$ for all $k<m_2$ and $h_{a^\star}^{m_1}(L_1)=\, ]c^{(l)}, c^{(l+1)}[$ for some $l\in \{1, ..., q\}$. 
\end{lemma}
\begin{proof}
The proof follows the same lines to those of  Lemma 3.2 of \cite{OW2010}.
\end{proof}

Using Lemmas \ref{lemma1} and \ref{lemma2},  for $\delta>0$ sufficiently small, the map $h_{{a^\star}}$ admits a collection $J_\delta$ such that all vertices of  the directed graph $\mathcal{P}(J_\delta)$ are completely accessible. By Theorem \ref{th_review2}, we may conclude that   for every $\alpha>0$ sufficiently small and for every $\hat{a}\in \Delta(\lambda, \alpha)$ close to ${a^\star}$, there exists a sequence $(a_n)_{n\in \NN}$ converging to $\hat{a}$ for which $h_{a_n}$ admits a superstable sink. 
By \eqref{sequence2}, we have  $$\mu_n= \exp \left(\frac{a_n-2n\pi}{\omega K_F }\right), \quad n\in \NN. $$ 
It is easy to see that $\dpt \lim_{n\rightarrow +\infty}\mu_n=0$. 
 Setting  $\mu_{1,n}=\mu_n$, $n\in \NN$, we obtain  the sequence needed to prove  Theorem  \ref{thm:C}.   The way this superstable periodic orbit (which is a \emph{critical point} of $h_{a_n}$) is obtained has been sketched in Figure \ref{periodic_orbit}.

  \begin{figure}[h]
\begin{center}
\includegraphics[height=5.0cm]{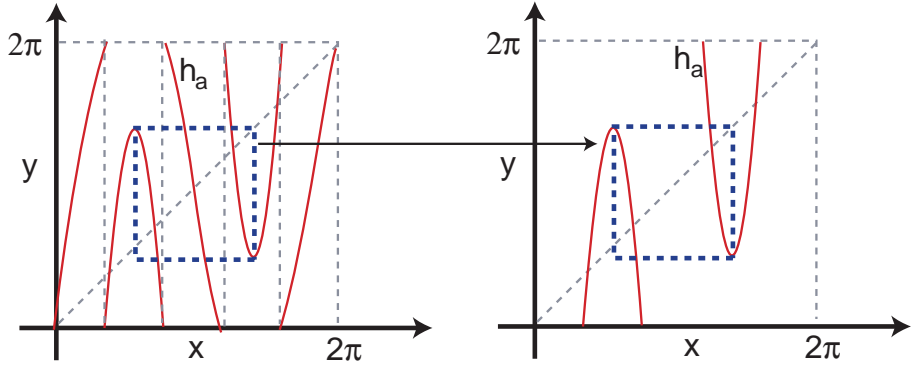}
\end{center}
\caption{\small  Graph of the map $h_a$ for $a=\mu_n$ and $\omega \gg 1$ with $q=2$ (number of critical points). Indicated is a superstable periodic orbit of period 2. }
\label{periodic_orbit}
\end{figure}

\section{Proof of  Theorem \ref{tangency1}}
\label{s: proof Th D}
This case reports the scenario described by Hypotheses \textbf{(P1)--(P5)} and \textbf{(P6b)--(P7b)} with $\mu_1=0$ and  $\mu_2>0$. We use the first return map $\mathcal{G}_{(0,\mu_2)} \mapsto \mathcal{G}_\mu $ (cf. Table 3). 
The proof of Theorem \ref{tangency1} follows the arguments of \cite{LR2017, RLA} and Appendices B and C of \cite{Wang2}. For the sake of completeness, we reproduce the main ideas of the proof.  It was shown  in   \cite{RLA}, that $\Lambda(\mathcal{G}_{(0, \mu_2)})$  contains infinitely many horseshoes near the heteroclinic network which emerges near $W^u(\CC_2)\pitchfork W^s(\CC_1)$. 
We describe the geometry of  $W^u(\CC_2)\cap \In(\CC_1)$ and $W^s(\CC_1)\cap \Out(\CC_2)$ for  $\mu_1=0$ and $\mu_2 \neq 0$. First, we introduce the notation depicted in Figure \ref{spirals1}: \\
\begin{itemize}
\item 
$(O_2^1,0)$ and $(O_2^2,0)$ with $0<O_2^1<O_2^2<2\pi$ are the coordinates of the two points
 where $W^u_\loc(\CC_2)$ meets $W^s_\loc(\CC_1)$ in $\Out(\CC_2)$ in the first turn around the cycle;\\
\item
$(I_1^1,0)$ and $(I_1^2,0)$  with $0<I_1^1<I_1^2<2\pi$  are the coordinates of the two  points  
where $W^u_\loc(\CC_2)$ meets $W^s_\loc(\CC_1)$ in $\In(\CC_1)$ in the first turn around the cycle; \\
\item
$(I_{1}^i,0)$ and $(O_2^i,0)$ are on the same trajectory for each $i \in \{1,2\}$ (also called $0$-pulses).\\
\end{itemize}

By \textbf{(P6b)}, for $\mu_1=0$ and small $\mu_2>0$, the curves $W^s_{\loc}(\CC_1)\cap \Out(\CC_2)$ and $W^u_\loc(\CC_2)\cap \In(\CC_1)$
are the  graphs of periodic functions $\eta_s$ and $\eta_u$, for which we  make the following conventions (see the meaning of $C_1$ and $C'_2$ in Subsection \ref{coords}): \\

\begin{itemize}
\item
$ W^s_{\loc}(\CC_1)\cap \Out(\CC_2)$ 
is the graph of   $\eta_s: \EU^1 \rightarrow [-1/C_2', C'_2]$;\\
\item
$   W^u_{\loc}(\CC_2)\cap \In(\CC_1)$ is the graph
  of  $\eta_u: \EU^1 \rightarrow [-1/C_1, C_1]$; \\
\item
   $\eta_u'(I_1^1), \eta_s'(O_2^1)>0$ and $\eta_u'(I_1^2), \eta_s'(O_2^2)<0$. \\
\end{itemize}

The maximum value of $\eta_u$ depends on $\mu$ and is  attained at some point of the type:
$$
\left(y_2^{(1)}, \theta^{(1)}\right)= ( \theta^\star(\mu), M) \qquad \text{with} \qquad I_1^1<\theta^\star <I_1^2 \qquad \text{and}\qquad 0<M<C_1<1.
$$

We will need to  introduce the  definition of spiral on an annulus $\mathcal{A}$  parametrised by  the coordinates $(y, \theta)\in [0, C_1]\times \EU^1$.

\begin{definition}
\label{spiral_def}
A \emph{spiral} on the annulus $\mathcal{A}$ \emph{accumulating on the circle} defined by  $y=0$ is a curve on $\mathcal{A}$, without self-intersections, that is the image, by the parametrisation $(y,\theta )$, of a $C^1$ map
$H:(b,c)\rightarrow [0,C_1]\times \EU^1$, 
$$
H(s)=\left(y(s), \theta(s) \right),
$$
 such that:\\
\begin{enumerate}
\renewcommand{\theenumi}{\roman{enumi}}
\renewcommand{\labelenumi}{{\theenumi})}
\item \label{monotonicity}
there are $\tilde{b}\le \tilde{c}\in (b,c)$ for which both $\theta(s)$ and $y(s)$ are monotonic in each of the intervals $(b,\tilde{b})$ and $(\tilde{c},c)$; \\
\item \label{turns}
either $\dpt \lim_{s\to b^+}\theta(s)=\lim_{s\to c^-}\theta(s)=+\infty$ or $\dpt \lim_{s\to b^+}\theta(s)=\lim_{s\to c^-}\theta(s)=-\infty;$ \\
\item \label{accumulates}
$\dpt \lim_{s\to b^+}y(s)=\lim_{s\to c^-}y(s)=0$.
\end{enumerate}
\end{definition} 

\begin{figure}[ht]
\begin{center}
\includegraphics[height=3.8cm]{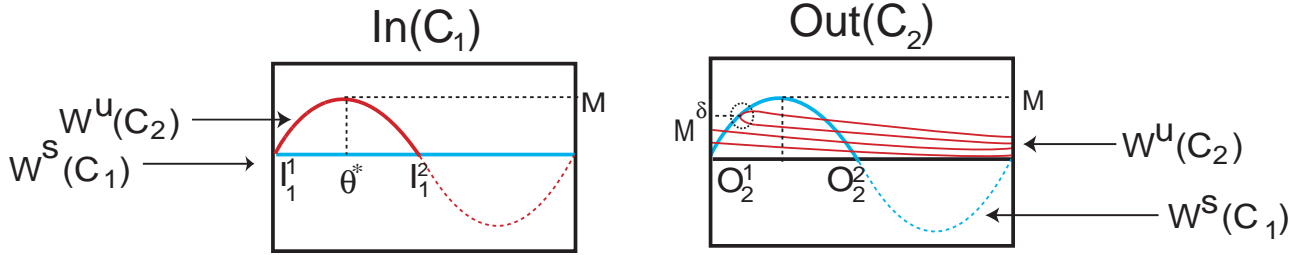}
\end{center}
\caption{\small The set $\eta_u=W^u(\CC_2)\cap \In^+ (\CC_1)$ is mapped by ${Loc}_2\circ \Psi_{1\rightarrow 2}\circ  Loc_1$ into a spiral accumulating on the circle defined by $\overline{y}_1^{(2)}=0$ (\emph{ie} the parametrisation of $\Out(\CC_2)\cap W^u(\CC_2)$).  There is a sequence $(\mu_i)_i$ for which the flow of $\mathcal{G}_{\mu_i}$ exhibits a quadratic heteroclinic tangency. }
\label{spirals1}
\end{figure} 

It follows from the assumptions on the function $\theta(s)$ that it has either a global minimum or a global maximum, and that $y(s)$ always has a global maximum.
The point where the map $\theta(s)$  has a global minimum or a global maximum will be called a \emph{fold point} of the spiral.
The global maximum value of $r(s)$ will be called the \emph{maximum radius} of the  spiral.

\begin{lemma}[Prop. 10 of \cite{LR2017}, adapted]
\label{lemma_auxiliar}
The map $ {Loc}_2\circ \Psi_{1\rightarrow 2}\circ  Loc_1$ transforms the line $W^u(\CC_2)\cap \In ^+(\CC_1)$ \footnote{This line is part of the graph of $\eta_u$ (red continuous line on the left image of Figure \ref{spirals1}).}  into a spiral on $Out(\CC_2)$ accumulating on the circle defined by 
$\Out(\CC_2)\cap W^u(\CC_2)$ (i.e., $\overline{y}_1^{(2)}=0$). This spiral has maximum radius $M^\delta$ as $\mu$ tends to zero. It has a fold point that turns around $W^u(\CC_2)\cap \Out^+(\CC_2)$. 
\end{lemma}
The geometric idea of Lemma \ref{lemma_auxiliar} is depicted in Figure \ref{spirals1}.  
Since $\delta>1$ and $M<1$, then  $M^\delta<M<C_1$.
\subsection{Proof of  Theorem \ref{tangency1}(1): rotational horseshoes}

For $\tau>0$ sufficiently small, define a rectangle $R\subset \Out(\CC_2)$ parameterized by $[O_2^1-\tau, O_2^2+ \tau]\times ]0, \tau]$. The map $\mathcal{G}_{\mu}$ compresses $R$ in the vertical direction and stretches in the vertical direction, making the image infinitely long towards both ends. In other words, the map $\mathcal{G}_{\mu}$ folds and wraps $R$ infinitely many times. This is why $\Lambda (\mathcal{G}_{\mu})$ contains a horseshoe of infinitely many branches for all $\mu\in [0, \varepsilon]$. When restricted to a compact set of $\Out^+(\CC_2)$ not containing $W^u(\CC_2)$, these horseshoes are uniformly hyperbolic. All details of this proof may be found in  \cite[Th. 8]{RLA}\footnote{The shift dynamics is what the authors of \cite{RLA} call \emph{horseshoe in time}.}

\subsection{Proof of  Theorem \ref{tangency1}(2): sequence of heteroclinic tangencies}

The set $\eta_s$ divides $\Out(\CC_2)$ in two connected components.
By Lemma \ref{lemma_auxiliar} and using the fact that $M^\delta<M$,  the fold point of  ${Loc}_2\circ \Psi_{1\rightarrow 2}\circ  Loc_1(W^u(\CC_2)\cap \In^+(\CC_1))$ moves around $\Out(\CC_2)$ at a speed (with respect to the parameter $\mu$) greater than the speed of $\eta_s$, from one connected component to the other, as $\mu$ vanishes. 
Two points where the fold point of the  spiral ${Loc}_2\circ \Psi_{1\rightarrow 2}\circ  Loc_1(W^u(\CC_2)\cap \In^+(\CC_1))$   intersects the graph of $\eta_s$  come together and collapse  at the heteroclinic tangency associated to the two-dimensional manifolds $W^s(\CC_1)$ and $W^u(\CC_2)$.  As $\mu$ goes to zero, it creates a sequence $(\mu_i)_{i\in\NN}$ of tangencies to the graph of $\eta_s$ (see Figure \ref{spirals1}).  This tangency  has a quadratic form.

 \subsection{Proof of  Theorem \ref{tangency1}(3): strange attractors }
 
 
By  \cite{RLA} there is a   horseshoe near the $0$-pulses associated to $W^u(\CC_2)\pitchfork W^s(\CC_1)$.
Hence, there are hyperbolic fixed points of the first return map $\mathcal{G}_\mu$ arbitrarily close to  the $0$-pulse; let  $P_i$ be one of these periodic orbits.

First, note that $P_i$ is dissipative (the absolute value of the product of the Lyapunov multipliers is less that the unit).  The unstable manifold of $P_i$ crosses $W^s(\CC_1)$ and so its image under ${Loc}_2\circ \Psi_{1\rightarrow 2}\circ  Loc_1$ accumulates on $W^u(\CC_2)$ (some reasoning of Lemma  \ref{lemma_auxiliar});  in particular, ${Loc}_2\circ \Psi_{1\rightarrow 2}\circ  Loc_1(W^u(P_i)\cap \Out(\CC_2))$ contains  infinitely many spirals in $\Out(\CC_2)$, each one having  a fold point.
Since the fold points turn around $\Out(\CC_2)$ infinitely many times as $\mu$ varies,
this curve is tangent (quadratic tangency) to $W^s(P_i)$ at a sequence $\mu_i$ of values of $\mu$. Hence, there exists a sequence of parameter values for which the associated flow exhibits non-degenerate heteroclinic tangencies formed by the invariant manifolds of the periodic orbits of the horseshoe. The existence of strange attractors of H\'enon-type follows from \cite{MV93}. Setting  $\mu_{2,i}=\mu_i$, $i\in \NN$, we obtain  the sequence needed to prove  Theorem \ref{tangency1}(3). 
  
  \subsection{Proof of  Theorem \ref{tangency1}(4): sequence of sinks}
  
  The existence of a sequence of parameter values for which the flow of \eqref{general_R3} exhibits a sink is a result of the quadratic homoclinic tangency associated to a dissipative point. The result follows from Gavrilov-Shilnikov \cite{GavS} and Newhouse \cite{Newhouse79} theories.

\section{Proof of Proposition \ref{thm:G}}
\label{Proof_ThE}

This case addresses the scenario described by $\mu_1=0$ and  $\mu_2>0$. This is why we use the first return map $\mathcal{G}_{(0, \mu_2)} \mapsto \mathcal{G}_\mu $ (cf. Table 3). 
Taking into account Lemma \ref{return2}, the singular cycle associated to $\mathcal{G}_{ \mu}$ has the form:
\begin{equation}
\label{singular2}
 h_a\left(\theta \right)= \theta+ \xi  + \mu_2\psi_2\left( \theta\right) -
  a   - \omega K_G  \ln \left[  \phi_2\left(   \theta\right)\right] , \qquad \theta \in \EU^1,
\end{equation}
where $$\xi=\xi_1+\xi_2 \qquad \text{and} \qquad K_G = \frac{1}{e_1+\omega_2^{(1)}} +\frac{ \delta_1}{e_2+\omega_2^{(2)}}.$$ The meaning of $a$ is the same of equation \eqref{sequence2}.
Since $\phi_2$ has zeros, the singular limit has logarithmic singularities. 
By \cite[pp. 534]{TW12}, for  $\varepsilon>0$ small and  $\omega \gg 1$, there exists a set $\Delta \subset \EU^1$ of $a$-values with $Leb_1(\Delta)>0$ such that if $a\in \Delta$ and $c\in C$ (set of critical points of $h_a$), the following inequality holds:
\begin{equation}
\label{limsup}
\forall n\in \NN, \qquad |(h_a^n)'(h_a(c))|>(\omega K_G )^{\lambda n}.
\end{equation}
In addition $\dpt\lim_{\omega \rightarrow +\infty}Leb_1(\Delta)=2\pi$. Since the recurrence of the critical points is almost inevitable with respect to the Lebesgue measure (\cite{TW12}), from \eqref{limsup} we may conclude that  for almost all points $\theta\in \EU^1$, we have:
$$
\limsup_{n\rightarrow +\infty} \frac{\ln |(h_a^n)'(\theta)|}{n} \geq \frac{\ln \omega}{\log \varepsilon} +\mathcal{O}(1) >0,
$$
where $\mathcal{O}(1)$ denotes the usual \emph{Landau notation}. This proves Proposition \ref{thm:G}.

\begin{remark}
The nature of the strange attractor of Proposition \ref{thm:G} is different from that of Theorem \ref{tangency1}. In the latter case, the strange attractor is of H\'enon-type and its basin of attraction is confined to a portion of the phase space near the homoclinicity. In the first case, the strange attractor shadows the entire locus of a two-dimensional torus (defined by $W^u(\CC_2)$). This strange attractor coexists with the heteroclinic pulses whose existence is guaranteed in \cite{RLA}.
\end{remark}
 
\section{Proof of Theorem \ref{mixture}}
\label{Proof_ThF}
We are looking for  homoclinic tangencies to $\CC_1$. In $\Out(\CC_1)$, the local unstable manifold of $\CC_1$ is parametrised by 
$$
\left(0, \overline\theta^{(1)}\right), \qquad \overline\theta^{(1)}\in \EU^1
$$

The expression for $ \Psi_{2\rightarrow 1} \circ \emph{Loc}\, _2  \circ \Psi_{1\rightarrow 2} (W^u_\loc(\CC_1)\cap \Out(\CC_1))\subset \In(\CC_1)$ is given by:

\begin{equation}
\label{homoC}
 \overline{y}_2^{(1)}  =  b_2 \frac{\mu_1^{\delta_2}}{\|\mu\|} \phi_1\left(0, \theta^{(1)}\right)^{\delta_2} + \frac{\mu_2}{\|\mu\|} \phi_2\left (\overline{y}_1^{(2)}, \theta^{(2)}\right)
\end{equation}

where
\begin{eqnarray*}
\overline{y}_1^{(2)} &=&  \frac{\mu_1^{\delta_2}}{\|\mu\|} \phi_1\left(0, \theta^{(1)}\right)^{\delta_2}\\
\overline{\theta}^{(2)} &=& \overline{\theta}^{(1)} + \xi_1 +\mu_1 \psi_1\left(0, \overline{\theta}^{(1)}\right) - \frac{\omega}{e_2+w_2^{(2)}}\ln \left(\mu_1 \phi_1\left(0, \overline{\theta}^{(1)}\right)\right)
\end{eqnarray*}

There is a homoclinic cycle to $\CC_1$ when $ \overline{y}_2^{(1)}=0$ (parameterization of the  local stable manifold of $\CC_1$). Using \eqref{homoC}, we get
$$
\text{Hom:}\qquad b_2 \mu_1^{\delta_2} C_1 + \mu_2C_2=0 \Leftrightarrow |\mu_2|= C |\mu_1|^{\delta_2}, \qquad C, C_1, C_2\in \RR^+.
$$

The invariant manifolds associated to $\CC_1$ develop a tangency along the curve \emph{Hom}, and  above this tangency (in the parameter space) there are transverse heteroclinic connections and thus a heteroclinic tangle. In the convex region defined by the curve \emph{Hom}, the set $W^u(\CC_1)$ does not play any longer the role of separatrix.   A plausible  bifurcation diagram is depicted in  Figure \ref{bd1} and an interpretation for the different regions   follows:
 \begin{eqnarray*}
\textbf{I} \quad &\to& \quad \text{Attracting two-dimensional torus if $\omega\approx 0$};  \\
\textbf{I} \quad &\to& \quad \text{Rank-one attractors   if $\omega\gg 1$}; \\
\textbf{II} \quad &\to& \quad \text{Dynamical structures associated to the torus-breakdown scenario when $\omega\approx 0$ }\\ 
  \quad & & \quad \text{and  heteroclinic bifurcations};\\ 
\textbf{III} \quad &\to& \quad \text{Heteroclinic tangles;  pulses; ``large'' strange attractors  if $\omega\gg 1$}.
\end{eqnarray*}

Dynamical properties of these three regions are discussed in Section \ref{s: discussion}.

\section{Rewriting Properties \textbf{(P7a)} and \textbf{(P7b)} with Melnikov integrals }
\label{s: Melnikov}
We start with the Melnikov functions for system \eqref{general_R2_perturbed}, explicitly defined for the unperturbed heteroclinic solutions $\ell_1$ and $\ell_2$ respectively. Let:
$$
\tau_1(t)=\frac{1}{|\ell_1'(t)|} l'_1(t) \qquad \text{and} \qquad \tau_2(t)=\frac{1}{|\ell_2'(t)|}l'_2(t) \
$$
be the unit tangent vectors of the heteroclinic solutions at $\ell_1$ and $\ell_2$ respectively. It is easy to check that:
$$
\lim_{t\rightarrow-\infty} \tau_1(t)=\overline{u}(e_1), \qquad \lim_{t\rightarrow+\infty} \tau_1(t)=\overline{u}(c_2), 
$$
and
$$
\lim_{t\rightarrow-\infty} \tau_2(t)=\overline{u}(e_2), \qquad \lim_{t\rightarrow+\infty} \tau_2(t)=-\overline{u}(c_1).
$$

For $i\in \{1, 2\}$, let $\tau_i^\perp(t) $  denote a unit vector that is perpendicular to $\tau_{\ell_i}(t)$ and $\left(\tau_{\ell_i}^\perp(t) \right)^T$  its transpose.
 The splitting of the stable and unstable manifolds on a global transverse cross section $\Sigma$ for the perturbed system  is measured by  the \emph{Melnikov function}:
$$
W_i(\theta)=\int_{-\infty}^{+\infty} \left\langle(P(\ell_i(t), t+\theta), Q(\ell_i(t), t+\theta)),\tau_i^\perp(t) \right\rangle \exp \left( -\int_{0}^tE_i(s)ds  \right) dt,\quad  \theta\in \EU^1
$$
where $\left\langle.,.\right\rangle$ denotes the usual inner product\footnote{This inner product corresponds to the \emph{wedge product}  $\wedge$ defined by Melnikov \cite{Melnikov}  (see also \cite{GH}).} in $\RR^2$ and
$$
E_i(t) = \tau_{\ell_i}^\perp(t) 
 \left( 
  \begin{array}{cc} \dpt
\frac{\partial  g_1 }{\partial x}(\ell_i(t))& \dpt\frac{\partial  g_1 }{\partial y}(\ell_i(t))\\
\dpt\frac{\partial  g_2 }{\partial x}(\ell_i(t)) &\dpt \frac{\partial  g_2 }{\partial y}(\ell_i(t))\\
\end{array} 
\right)
\left(\tau_i^\perp(t) \right)^T \in \RR.
$$

It is easy to check that 
$$
\lim_{t\rightarrow - \infty} E_1(t)= e_1  \qquad \text{and} \qquad \lim_{t\rightarrow + \infty} E_1(t)= c_2
$$
and
$$
\lim_{t\rightarrow - \infty} E_2(t)= e_2  \qquad \text{and} \qquad \lim_{t\rightarrow + \infty} E_2(t)= e_1.
$$
 
 \bigbreak
According to \cite{Bertozzi, GH, Melnikov}, Hypotheses  \textbf{(P7a)} and \textbf{(P7b)} main be rephrased as:
 \bigbreak
\begin{enumerate}
 \item[\textbf{(P7a)}] If $\mu_1>0$ and $\mu_2=0$, then $\dpt \min_{\theta \in \EU^1} W_1(\theta)$ and $\dpt \max_{\theta \in \EU^1} W_1(\theta)$ have the same sign. 
  \medbreak 
 \item[\textbf{(P7b)}]  If $\mu_2>0$ and $\mu_1=0$, then $\dpt \min_{\theta \in \EU^1} W_2(\theta)<0<\max_{\theta \in \EU^1} W_2(\theta)$ and if $W_2(\theta_0)=0$ for some $\theta_0\in \EU^1$, then $W_2'(\theta_0)\neq 0$.
\end{enumerate}

\section{Rank-one strange attractors and heteroclinic tangles: \\ a short discussion}
\label{s: discussion}
When a planar heteroclinic cycle associated to two dissipative saddles is periodically perturbed, the perturbation either pulls the stable and the unstable manifolds of the equilibria completely apart, or it creates chaos through a heteroclinic tangle. In both (exclusive) cases, the singular limit induced by the perturbed equation (at least $C^4$) in the extended phase space may be written as a family of two-dimensional maps. The   \emph{singular limit cycle}   is a one-dimensional map of the form:
$$
\theta \mapsto \theta + a + \omega K \ln |\Phi(\theta)|, \qquad \theta \in \EU^1
$$
where:
\begin{enumerate}
\item $K>0$ depends on the eigenvalues of the derivative of the original vector field at the hyperbolic equilibria;\\
\item $\omega$ is the frequency of the non-autonomous perturbation;\\
\item  $a \in \EU^1$ (depends on the magnitude of the forcing);\\
\item $\Phi: \EU^1 \to \RR$ is $C^3$ and periodic;\\
\item  $\Phi'(\theta)\neq 0$ if $\Phi(\theta)=0$ and $\Phi''(\theta )\neq 0$ if $\Phi'(\theta)=0$.\\
\end{enumerate}
 Usually, the map $\Phi$ may be seen as the \emph{classical Melnikov function}. 
\medbreak
For system \eqref{general_R3}, if  $\mu_1> 0 $ and $\mu_2=0$, the stable and unstable manifolds of the perturbed saddles are pulled completely apart by the forcing function, implying $\Phi (\theta)\neq 0$ for all $\theta \in \EU^1$. In this case, we obtain an attracting two-dimensional torus or strange attractors, to which the theory of rank-one maps may be applied. Here, the parameter $\omega$ plays an important role to understand how ``large'' strange attractors come from the destruction of an attracting  two-dimensional torus.

If  $\mu_1= 0 $ and $\mu_2>0$, the two-dimensional stable and unstable manifolds of the saddles intersect ($\Leftrightarrow \Phi(\theta) = 0$ has solutions)  and strange attractors are associated to a heteroclinic tangle. 
As $\omega$ gets larger, the contracting region gets smaller and the dynamics is more and more expanding in most of the phase space. The recurrence of the critical points is inevitable, and infinitesimal changes of dynamics occur when $a$ is varied. The logarithmic nature of the singular set turns out to present a new phenomenon which is unknown to occur for Misiurewicz-type maps. Proposition \ref{thm:G} states that strange attractors with \emph{nonuniform expansion} prevails  provided $\omega\gg 1$. When $\mu_1, \mu_2\neq 0$, we also proved   that, under conditions \textbf{(P1)--(P7)}, the existence of heteroclinic tangles is a \emph{prevalent phenomenon} for the dynamics of  \eqref{general_R3}.


The techniques we have used to prove the main results follow the spirit of previous results in the literature. In Table 4, we give an overview of the results   and the contribution of the present article (in blue) for the dynamics of \eqref{general_R3}.

\begin{table}[htb]
\small
\begin{center}
\begin{tabular}{|c|l|l|} \hline   && \\
{Configuration}  & \qquad  $W^u(\CC_1) \pitchfork W^s(\CC_2)$ \qquad  \qquad &\qquad  $W^u(\CC_1) \cap W^s(\CC_2)=\emptyset$ \qquad \\  && \\
\hline \hline
&&\\
 &Case 1 &Case 2 \textcolor{blue}{(Novelty of this article)}\\ 
  & Heteroclinic tangles &  Region with torus if $\omega \approx 0$  \\ 
  $W^u(\CC_2) \pitchfork W^s(\CC_1)$&  (Horseshoes, tangencies, sinks &  Region with rank-one attractors if  $\omega \gg 1$ \\ 
    & Newhouse phenomena, pulses) &  Superstable sinks \\ 
    & \cite{ChOW2013, LR2017, RLA}&  Heteroclinic tangles prevail  \\ 
 && \\
 \hline \hline  &&\\
 &Case 3 & Case 4 \\ && Region with torus if $\omega \approx 0$\\ 
$W^u(\CC_2) \cap W^s(\CC_1)=\emptyset$ &Similar to Case 2 &  Region with rank-one attractors if  $\omega \gg 1$ \\
&  & \cite{Mohapatra} \\

&& \\ \hline

\hline

\end{tabular}
\end{center}
\label{notationA}
\bigskip
\caption{\small Overview of the results in the literature and the contribution of the present article (in blue) for the dynamics of \eqref{general_R3}. }
\end{table}

Strange attractors found in Theorems \ref{thm:F} and  \ref{tangency1} are qualitative different. In the first case, they are rank-one strange attractors; if  $\omega \gg 1$, they are not confined to a small portion of the phase space -- their basin of attraction spreads around the whole ``torus-ghost" (annulus in the cross section $\Out(\CC_1)$). According to \cite{BST98}, they are called ``large'' strange attractors.
On the other hand, in Theorem  \ref{tangency1},  H\'enon-type strange attractors are confined to a small portion of the phase space near the homoclinic tangency. In this case, the study of  properties of the strange attractors is more involved due to the existence of infinitely many pulses which cannot be disconnected from the attractor  (there are infinitely many points within $W^s(O_1)$ where the first return map is not well defined).  This difference justifies the title of this manuscript. 
 A lot more needs to be done before these two types of chaos are well understood.

The sinks of Theorems \ref{thm:C} and \ref{tangency1}(3) have similarities but they have been obtained in a different way. While in the first case,  sinks are due to critical periodic points of the singular cycle \cite{OW2010}, in the second, sinks are a consequence of Gavrilov-Newhouse phenomena \cite{GavS, Newhouse79}. 

Finally, we would like to point out that in Case 4,   the non-wandering set associated to $\Gamma$  has one, two or three attracting tori according to the relative position of $W^u(\CC_1)$ and $W^u(\CC_2)$.   In the same spirit of \cite{ChOW2013}, in Figure \ref{other_cases1}, we have summarized all possibles of invariant curves that can appear in the unfolding of   system \eqref{general_R3}, for  $\mu_1\neq 0 $ and $\mu_2=0$.  In these cases, Theorems \ref{thm:B}, \ref{thm:F} and \ref{thm:C} still hold with minor variations.
Finding an explicit example where Hypotheses \textbf{(P1)--(P7)} are met is the next ongoing research. 

 \begin{figure}[h]
\begin{center}
\includegraphics[height=8.1cm]{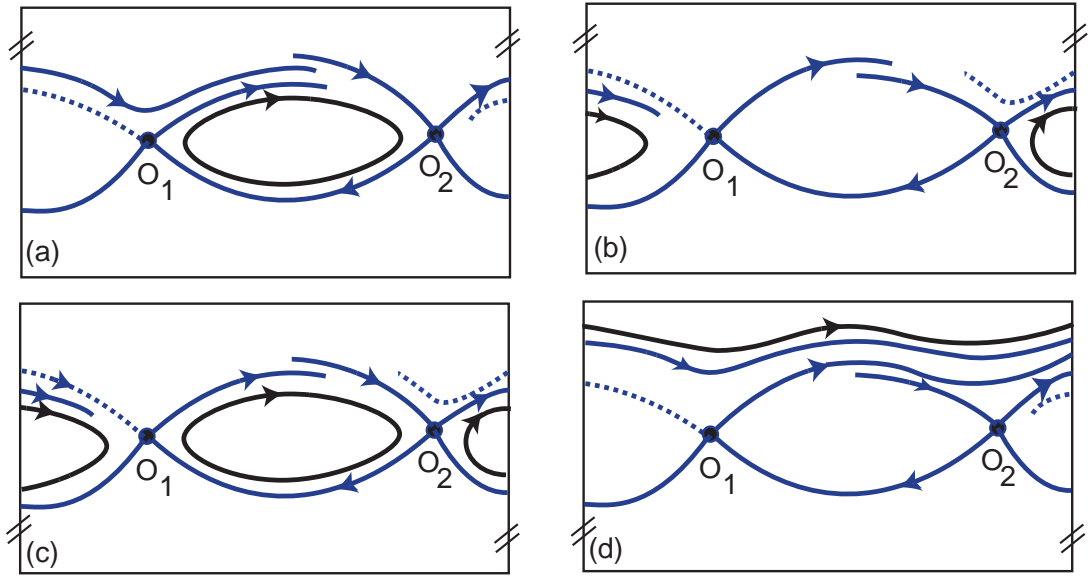}
\end{center}
\caption{\small  Four types of configuration for \eqref{general_R2_perturbed} when $\mu_1>0$ and $\mu_2=0$. (a), (b): one contractible periodic solution; (c): two contractible periodic solutions;  (d) one non-contractible periodic solution. Double bars mean that the sides are identified. }
\label{other_cases1}
\end{figure}

\appendix
\section{Notation}
\label{app1}
In Table 5, we list the main notation for constants and auxiliary functions used in this paper in order of appearance with the reference of the section containing a definition.

\begin{table}[htb]
\begin{center}
\begin{tabular}{|c|l|c|} \hline 
{Notation}  & \qquad  Definition/meaning \qquad  \qquad &\quad  Section \qquad \\
\hline \hline
&&\\
$\mathcal{V}$ & Open region of $\RR^2$ where equation \eqref{general_R2} is well defined  & \S \ref{ss:setting}  \\  &&\\
 \hline \hline 
 &&\\

$ O_1, O_2$  &  Saddle-equilibria of the equation \eqref{general_R2}   & \S \ref{ss:setting}  \\  &&\\
 \hline \hline 
 &&\\

$ {\ell}_1,  {\ell}_2$  & Connections from $O_1$ to $O_2$ and from $O_2$ to $O_1$    & \S \ref{ss:setting}  \\  &&\\

 \hline \hline 
 &&\\

$\mathcal{A}$ & Region limited by the heteroclinic cycle $ {\ell}_1\cup {\ell}_2$  & \S \ref{ss:setting}   \\  &&\\
 \hline \hline  
 &&\\
$\mathcal{V}^\star$ & Inner basin of attraction of the heteroclinic cycle $ {\ell}_1\cup {\ell}_2$  & \S \ref{ss:setting}  \\  &(absorbing domain)&\\
 \hline \hline  &&\\
 $\mathbb{V}$ & $\mathcal{V}\times \EU^1$ -- open region where equation \eqref{general_R3} is defined  & \S \ref{ss:lift}  \\  &&\\
 \hline \hline  &&\\
 $ \CC_1, \CC_2$  & Saddle periodic solutions of the equation \eqref{general_R3}   & \S \ref{ss:lift}  \\  &&\\
 \hline \hline  &&\\
 $ \Gamma$  & Heteroclinic cycle associated to  $ \CC_1, \CC_2$   & \S \ref{ss:lift}  \\  &&\\
 \hline \hline 
 &&\\
$\mathbb{A}$ & $\mathcal{A}\times \EU^1$  & \S \ref{ss:lift}  \\  &&\\
 \hline \hline  &&\\
 $\mathbb{V}^\star$ &  $\mathcal{V}^\star\times \EU^1$  & \S \ref{ss:lift}  \\  &&\\
 \hline \hline  &&\\
 $ \mathcal{L}_1, \mathcal{L}_2$  &   Connections from $\CC_1$ to $\CC_2$ and from $\CC_2$ to $\CC_1$   & \S \ref{ss:lift}  \\  &&\\

 \hline \hline 
 &&\\
$V_1, V_2$ & Hollow cylinders around $\CC_1$ and $\CC_2$  & \S \ref{ss:lift}  \\  &&\\
 \hline \hline  &&\\  $A\equiv_\mathbb{V^\star} B$  & The manifolds $A$ and $B$   coincide within $\mathbb{V}^\star$  & \S \ref{ss:parameters}  \\  &&\\
    \hline \hline  &&\\  $\mathcal{F}_{(\mu_1, \mu_2)}\equiv \mathcal{F}_{\mu}$  & First return map to $\Out(\CC_1)$  & \S \ref{ss:simplified}     and Table 3\\  &&\\
    \hline \hline  &&\\  $\mathcal{G}_{(\mu_1, \mu_2)}\equiv \mathcal{G}_{\mu}$    & First return map to $\Out(\CC_2)$  & \S \ref{ss:simplified}  and Table 3  \\  &&\\
   
  \hline  
 \hline  
\end{tabular}
\end{center}
\label{notationB}
\bigskip
\caption{\small Notation.}
\end{table} 

\end{document}